\documentclass[a4paper,reqno]{amsart}

\usepackage[utf8]{inputenc}
\usepackage{amssymb}
\usepackage{enumitem}
\usepackage{color}
\usepackage{mathrsfs}
\usepackage{mathtools}
\setlist[enumerate]{label=\emph{(\roman*)}}

\usepackage[colorlinks=true]{hyperref}
\hypersetup{linkcolor=red,
     filecolor=black,   
     urlcolor=cyan,citecolor=cyan}

\usepackage{environ}

\newtheorem{theorem}{Theorem}[section]

\newtheorem{lemma}[theorem]{Lemma}

\theoremstyle{definition}

\newtheorem{remark}[theorem]{Remark}
\numberwithin{equation}{section}
\newcommand{\R}{\mathbb{R}}

\def\E{\dot{H}^{1}\times L^{2}}
\def\pshbb1#1#2{\left(#1,#2\right)_{\dot H^1_{{{\ell}}_1}}}

\def\RR{\mathbb{R}}
\def\E{\mathcal{H}}
\def\px{\partial_{x}}
\def\pt{\partial_{t}}
\def \vp{\varphi_{1}}
\def \vpp{\varphi_{2}}
\def \d {\,\mathrm{d}}
\begin{document}

\parindent=0pt

\title[conditional stability for 1D NLKG equation]{Conditional stability of multi-solitons for the 1D NLKG equation with double power nonlinearity}
\author{XU YUAN}

\address{CMLS, \'Ecole polytechnique, CNRS, Institut Polytechnique de Paris, F-91128 Palaiseau Cedex, France.}

\email{xu.yuan@polytechnique.edu}
\begin{abstract}
We consider the one-dimensional nonlinear Klein-Gordon equation with a double power focusing-defocusing nonlinearity
\begin{equation*}
\partial_{t}^{2}u-\partial_{x}^{2}u+u-|u|^{p-1}u+|u|^{q-1}u=0,\quad \mbox{on}\ [0,\infty)\times \RR,
\end{equation*}
where $1<q<p<\infty$. The main result
states the stability in the energy space $H^{1}(\RR)\times L^{2}(\RR)$ of the sums of decoupled solitary waves with different speeds, up to the natural instabilities.
The proof is inspired by the techniques developed for the generalized Korteweg-de Vries
equation and the nonlinear Schr\"odinger equation in a similar context by Martel, Merle and Tsai~\cite{MMTgkdv,MMTSchor}.
However, the adaptation of this strategy to a wave-type equation
requires the introduction of a new energy functional adapted to the Lorentz transform.
\end{abstract}
\maketitle
\section{Introduction}
\subsection{Main result}
We consider the one-dimensional nonlinear Klein-Gordon equation (NLKG) with a double power nonlinearity
\begin{equation}\label{NLKG}
\left\{ \begin{aligned}
&\partial_t^2 u - \partial_{x}^{2} u+u - |u|^{p-1}u+|u|^{q-1}u = 0, \quad (t,x)\in [0,\infty)\times \RR,\\
& u_{|t=0} = u_0\in H^1,\quad 
\partial_t u_{|t=0} = u_1\in L^2,
\end{aligned}\right.
\end{equation}
where $1<q<p<\infty$. This equation also rewrites as a first order system in time for the function $\vec{u}=(u_{1},u_{2})$,
\begin{equation}\label{NLKGvec}
\left\{
\begin{aligned}
&\pt u_{1}=u_{2}\\
&\pt u_{2}=\px ^{2} u_{1}-u_{1}+f(u_{1}),
\end{aligned}\right.
\end{equation}
where 
\begin{equation}\label{nonlin}
f(u_{1})=|u_{1}|^{p-1}u_{1}-|u_{1}|^{q-1}u_{1}.
\end{equation}
We recall that the Cauchy problem for equation~\eqref{NLKGvec}
is locally well-posed in the energy space $H^{1}\times L^{2}$. See \emph{e.g.}~\cite{GVKG}. Denote
$F(u_{1})=\frac{1}{p+1}|u_{1}|^{p+1}-\frac{1}{q+1}|u_{1}|^{q+1}$.
For any solution $\vec{u}=\left(u_{1},u_{2}\right)$ of \eqref{NLKGvec} in $H^{1}\times L^{2}$, the energy $E(\vec{u})$ and momentum $\mathcal{I}(\vec{u})$ are conserved, where
\begin{equation*}
E(\vec{u})=\int_{\RR}\left\{(\px u_{1})^{2}+u_{1}^{2}+u_{2}^{2}-2F(u_{1})\right\}\d x,\quad 
\mathcal{I}(\vec{u})=2\int_{\RR}(\px u_{1})u_{2}\d x.
\end{equation*}
Denote by $Q$ the \emph{ground state}, which is the unique positive even solution in $H^1$ of the equation
\begin{equation}\label{eq:Q}
Q''-Q+f(Q)=0\quad \mbox{on}\ \ \RR.
\end{equation}
The existence and properties of this solution are studied in~\cite[Section 6]{BL} (see also Remark 1.4 and Section~\ref{S:2.1}).
The ground state generates the standing wave solution $\vec{Q}=(Q,0)$ of~\eqref{NLKGvec}. 
Moreover, using the Lorentz transformation on $\vec Q$, one obtains traveling solitary waves or \emph{solitons}: for $\ell\in \RR$, with
$-1<\ell<1$, let 
\begin{equation*}
Q_{\ell}(x)=Q\left(\frac{x}{\sqrt{1-\ell^{2}}}\right),\quad 
\vec{Q}_{\ell}=\left(\begin{array}{c} Q_{\ell} \\
-\ell\px Q_{\ell}
\end{array}\right),
\end{equation*}
then $\vec{u}(t,x)=\vec{Q}_{\ell}(x-\ell t)$ is a solution of~\eqref{NLKGvec}.

It is well-known (see precise statements and references in~\S\ref{S:2.1}) that the operator
\begin{equation*}
\mathcal{L}=-\partial_{x}^{2}+1-f'(Q)
\end{equation*}
appearing after linearization of equation~\eqref{NLKGvec} around $\vec{Q}=(Q,0)$, has a unique negative 
eigenvalue $-\nu_{0}^{2}$ $( \nu_{0}>0)$, with corresponding smooth even 
eigenfunction~$Y$. 
Set
\begin{equation*}
\vec{Y}^{+}=\left(\begin{array}{c}
Y\\ \nu_{0}Y
\end{array}\right)\quad \mbox{and}\quad 
\vec{Z}^{+}=\left(\begin{array}{c}
\nu_{0}Y\\ Y
\end{array}\right).
\end{equation*}
From explicit computations, the function $\vec{u}^{+}(t,x)=\exp(\nu_{0}t)\vec{Y}^{+}(x)$
is solution of the linearized system
\begin{equation}\label{NLKGlin}
\left\{
\begin{aligned}
&\pt u_{1}=u_{2}\\
&\pt u_{2}=-\mathcal{L}u_{1}.
\end{aligned}\right.
\end{equation}
Since $\nu_{0}>0$, the solution $\vec{u}^{+}$ illustrates the (one-dimensional) exponential instability of the solitary wave $\vec{Q}$ in positive time. An equivalent formulation of instability is obtained by observing that for any solution $\vec{u}$ of~\eqref{NLKGlin}, it holds
\begin{equation*}
\frac{\d }{\d t}a^{+}=\nu_{0}a^{+}\quad \mbox{where}\ \ a^{+}(t)=\left(\vec{u}(t),\vec{Z}^{+}\right)_{L^{2}}.
\end{equation*}
More generally, for $-1<\ell<1$, set
\begin{equation*}
Y_{\ell}=Y\left(\frac{x}{\sqrt{1-\ell^{2}}}\right)\quad \mbox{and}\quad 
\vec{Z}^{+}_{\ell}=\left(
\begin{array}{c}(\ell\partial_{x}Y_{\ell}+ \frac{{\nu_{0}}}{\sqrt{1-\ell^{2}}}Y_{\ell})e^{ \frac{\ell{\nu_{0}}}{\sqrt{1-\ell^{2}}}x}\\
Y_{\ell}e^{ \frac{ \ell{\nu_{0}}}{\sqrt{1-\ell^{2}}}x}
\end{array}\right).
\end{equation*}
The main purpose of this article is to study the conditional stability of multi-solitons with different speeds for~\eqref{NLKG}.
More precisely, the main result is the following.
\begin{theorem}\label{main:theo}
Let $N\ge 2$. For all $n\in \{1,\ldots, N\}$, let $\sigma_{n}=\pm 1$, $\ell_{n}\in (-1,1)$ with
$-1<\ell_{1}<\ell_{2}<\cdots<\ell_{N}<1$. There exist $L_{0}>0$, $C_{0}>0$, $\gamma_{0}>0$ and $\delta_{0}>0$ such that the following is true.
Let $\vec{\varepsilon}\in H^{1}\times L^{2}$ and $y_{1}^{0}<\cdots<y_{N}^{0}$ be such that there exist $L>L_{0}$ and $0<\delta<\delta_{0}$
with
\begin{equation*}
\|\vec{\varepsilon}\|_{H^{1}\times L^{2}}<\delta\quad \mbox{and}\quad y_{n+1}^{0}-y_{n}^{0}>L \quad \mbox{for all $n=1,\ldots,N-1$}.
\end{equation*}
Then, there exist $h^{+}_{1},\ldots,h_{N}^{+}\in \mathbb{R}$ satisfying
\begin{equation*}
\sum_{n=1}^{N}|h^{+}_{n}|\le C_{0} \big(\delta+e^{-\gamma_{0}L}\big),
\end{equation*}
such that the solution $\vec{u}=(u_{1},u_{2})$ of~\eqref{NLKGvec} with initial data
\begin{equation*}
\vec{u}_{0}=\sum_{n=1}^{N}\big(\sigma_{n}\vec{Q}_{\ell_{n}}+h^{+}_{n}\vec{Z}^{+}_{\ell_{n}}
\big)(\cdot-y_{n}^{0})+\vec{\varepsilon}
\end{equation*}
is globally defined in $H^{1}\times L^{2}$ for $t\ge 0$ and, for all $t\ge 0$,
\begin{equation*}
\bigg\|\vec{u}(t)-\sum_{n=1}^{N}\sigma_{n}\vec{Q}_{\ell_{n}}(\cdot-y_{n}(t))\bigg\|_{H^{1}\times L^{2}}\le C_{0}\big(\delta+e^{-\gamma_{0}L}\big),
\end{equation*}
where $y_{1}(t),\ldots,y_{N}(t)$ are $C^{1}$ functions 
satisfying, for all $n=1,\ldots, N$, $t\ge 0$,
\begin{equation*}
|y_{n}(0)-y_{n}^{0}|\le C_{0}\big(\delta+e^{-\gamma_{0}L}\big),\quad |\dot{y}_{n}(t)-\ell_{n}|\le C^{2}_{0}\big(\delta+e^{-\gamma_{0}L}\big).
\end{equation*}
\end{theorem}
\begin{remark}
As remarked before, each soliton solution $\vec{u}(t,x)=\pm \vec{Q}_{\ell_{n}}(x-\ell_n t)$ of equation~\eqref{NLKGvec} has exactly one exponential instability direction in positive time.
Recall that such instability has led to the construction of co-dimension one center-stable manifolds for initial data close to a single soliton for the NLKG or the wave equations in~\cite{KNS,NSEMS,NS}. 
In the case of a perturbation of the sum of $N$ decoupled solitons $\sum_{n=1}^{N} \sigma_n Q_{\ell_{n}}(\cdot-y_{n})$, this justifies the need of adjusting the $N$ free parameters $(h^{+}_{n})_{n\in \{1,\ldots,N\}}$ in the choice of the initial data to obtain global existence and stability
in Theorem~\ref{main:theo}.
Such difficulty does not appear when the solitons are stable, like for the mass subcritical generalized Korteweg-de Vries (gKdV) and nonlinear Schr\"odinger equation (NLS) discussed in the next remark.
\end{remark}
\begin{remark}
Historically, multi-soliton solutions were first studied extensively for integrable equations, like the Korteweg-de Vries equation and the one-dimensional cubic Schr\"odinger equation.
In the non-integrable cases, for dispersive or wave equations, the first result concerning stability and asymptotic stability of multi-soliton solutions was given by Perelman~\cite{GP1} for some NLS equations, extending the single soliton case treated by Buslaev and Perelman~\cite{BP}. 
The present work is inspired by~\cite{MMTgkdv,MMTSchor} proving the stability of the sums of decoupled solitons for the mass subcritical gKdV equations and some NLS equations. See also~\cite{RSS} and~\cite{BGS,CP,LCW,MTX} for other related models.

Such stability results are closely related to the existence of \emph{asymptotic pure multi-solitons}, 
established in the case of the NLKG equation by C\^ote and Mu\~noz~\cite{CMkg}.
For gKdV, NLS, and the energy-critical wave equation for both stable and unstable solitons, we refer to~\cite{Com,CMM,Ma,MMnls,MMwave,Xwave}. 
\end{remark}
\begin{remark}
The double power nonlinearity \eqref{nonlin} is a typical choice for dispersive or wave equations, but 
Theorem~\ref{main:theo} can be extended to any general real-valued $C^{1,\alpha}$ function $f$
(with $\alpha>0$)
satisfying the following:
\begin{enumerate}
\item[$\rm{(i)}$] The function $f$ is odd and $f(0)=f'(0)=0$.
\item[$\rm{(ii)}$] There exists a smallest $s_{0}>0$ such that $F(s_{0})-\frac{1}{2}s_{0}^{2}=0$, and $f(s_{0})-s_{0}>0$.
\item[$\rm{(iii)}$] There exists $r_{0}>0$ such that for all $s\in (-r_{0},r_{0})$, $sf(s)-2F(s)\le0$.
\end{enumerate}
Here, $F(s)=\int_{0}^{s}f(\sigma)\d \sigma$ for $s\in \RR$.
Conditions (i)-(ii) are related to the necessary and sufficient condition for existence of a unique standing wave in~\cite[Theorem~5]{BL}. Note that for $f$ defined in~\eqref{nonlin},
the existence of a smallest $s_0>0$ such that $F(s_{0})-\frac{1}{2}s_{0}^{2}=0$ is clear
and then
\begin{align*}
f(s_0)-s_0 = s_0^{p} - s_0^{q} - s_0 
&=\frac1{s_0} \left((p+1) F(s_0) + \frac{p-q}{q+1} s_0^{q+1} - s_0^2\right)\\
&=\frac 1{s_0} \left(\frac{p-q}{q+1} s_0^{q+1} + \frac{p-1}2 s_0^2\right)>0,
\end{align*}
which justifies the existence of the standing wave $Q$ for \eqref{NLKG} (see also \cite[Example~2, p. 318]{BL}).

Condition (iii) ensures that the nonlinear effect is
defocusing for small values of $u_1$, which is decisive in our proof (more precisely, in the proof of the monotonicity result Lemma~\ref{le:virial}).
Recall that this issue is already present in~\cite{MMTSchor,GP1} for the NLS equation, where a weaker but similar restriction is imposed on the nonlinearity.
\end{remark}

To our knowledge, this article 
proves the first statement of (conditional) stability of the sums of decoupled solitons for a wave-type equation.
Observe that the Lorentz transform involved in the propagation of solitons has quite different properties than the Galilean transform for the NLS equation or the natural propagation phenomenon related to the solitons of the gKdV equation.
Compared to the energy method developed in~\cite{MMTSchor}, we use a similar functional based on localized versions of the momentum and the energy
around each soliton (respectively defined in \eqref{def:In} and \eqref{def:En}), but a specific combination of these quantities has to be used in the wave case (see \eqref{def:comE} and the definition of the coefficients $c_{n}$ in \eqref{def:ck}); moreover,
the introduction of an additional lower order term $F_{n}$ (see~\eqref{def:Fn}) was necessary to remove diverging terms.
See also Remark~\ref{re:cn}.

\smallskip

This paper is organized as follows.
Section~\ref{S:2} introduces technical tools involved in a dynamical approach to the $N$-soliton problem for~\eqref{NLKG}: estimates of the nonlinear interactions between solitons,
decomposition by modulation and parameter estimates. Energy estimates and monotonicity properties are proved in Section~\ref{Se:3}. Finally, 
Theorem~\ref{main:theo} is proved in Section~\ref{S:4}.
\subsection{Notation}
We denote $(\cdot,\cdot)_{L^{2}}$ the $L^{2}$ scalar product for real-valued functions $u,v\in L^{2}$,
\begin{equation*}
\left(u,v\right)_{L^{2}}:=\int_{\RR}u(x)v(x)\d x.
\end{equation*}
For
\begin{equation*}
\vec{u}=\left(\begin{array}{c}u_{1} \\u_{2}\end{array}\right),\quad \vec{v}=\left(\begin{array}{c}v_{1} \\v_{2}\end{array}\right),
\end{equation*}
denote
\begin{equation*}
\big(\vec{u},\vec{v})_{L^{2}}:=\sum_{k=1,2}\big(u_{k},v_{k})_{L^{2}},
\quad 
\|\vec{u}\|_{\E}^{2}:=\|u_{1}\|_{H^{1}}^{2}+\|u_{2}\|_{L^{2}}^{2}.
\end{equation*}
For $f\in L^{2}$ and $\ell\in (-1,1)$, set
\begin{equation*}
f_{\ell}(x)=f\left(x_{\ell}\right),\quad \mbox{where}\quad x_{\ell}=\frac{x}{\sqrt{1-\ell^{2}}}.
\end{equation*}

\subsection*{Acknowledgements}
The author would like to thank his advisor, Professor Yvan Martel, for his generous help,
encouragement, and guidance related to this work. The author is also grateful to the
anonymous referees for careful reading and useful suggestions, which led to an improved version of this article.

\section{Preliminaries}\label{S:2}

\subsection{Spectral theory}\label{S:2.1}
It is well-known that the unique positive even solution $Q$ of~\eqref{eq:Q} is smooth and satisfies the following estimate on $\R$,
\begin{equation}\label{Qdec}
|Q^{(\alpha)}(x)|\lesssim e^{-|x|} \quad \mbox{for any $\alpha\in \mathbb{N}$}.
\end{equation} 

In this section, we recall the spectral properties of the linearized operator around~$Q_{\ell}$.
First, for $-1<\ell<1$, let 
\begin{equation*}
 \mathcal{L}_{\ell}=-(1-\ell^{2})\partial_{x}^{2}+1-f'(Q_{\ell}).
\end{equation*}
 We recall the following standard spectral properties for $\mathcal{L}$ and $\mathcal{L}_{\ell}$ (see $e.g.$~\cite[Lemma~1 and Corollary 1]{CMkg}).
\begin{lemma}
\emph{(i) Spectral properties}.
The unbounded operator $\mathcal{L}$ on $L^{2}$ with
domain $H^{2}$ is self-adjoint, its continuous spectrum is $[1,+\infty)$,
its kernel is spanned by $Q'$ and it has a unique negative 
eigenvalue $-\nu_{0}^{2}$ $( \nu_{0}>0)$, with corresponding smooth even 
eigenfunction $Y$. Moreover, on $\RR$,
\begin{equation}\label{est:decayY}
\big|Y^{(\alpha)}(x)\big|\lesssim e^{-\sqrt{1+\nu^{2}_{0}}|x|}\quad \mbox{for any $\alpha\in \mathbb{N}$}.
\end{equation}

\emph{(ii) Coercivity property of $\mathcal{L}$.}
There exists $\nu>0$ such that, for all $v\in H^{1}$,
\begin{equation*}
\big(\mathcal{L}v,v\big)_{L^{2}}\ge \nu \|v\|_{H^{1}}^{2}-\nu^{-1}\big((v,Q')_{L^{2}}^{2}+(v,Y)_{L^{2}}^{2}\big).
\end{equation*}

\emph{(iii) Coercivity property of $\mathcal{L}_{\ell}$.}
Let $\ell\in (-1,1)$.
There exists $\nu>0$ such that, for all $v\in H^{1}$,
\begin{equation*}
\big(\mathcal{L}_{\ell}v,v\big)_{L^{2}}\ge \nu \|v\|_{H^{1}}^{2}-\nu^{-1}\big((v,\partial_{x}Q_{\ell})_{L^{2}}^{2}+(v,Y_{\ell})_{L^{2}}^{2}\big).
\end{equation*}
\end{lemma}
\smallskip

Second, we define
\begin{equation*}
\mathcal{H}_{\ell}=\left(\begin{array}{cc}-\partial_{x}^{2}+1-f'(Q_{\ell}) & -\ell\partial_{x} \\\ell\partial_{x} & 1\end{array}\right),\quad 
J=\left(\begin{array}{cc}0 & 1 \\ -1 & 0\end{array}\right),
\end{equation*}
and
\begin{equation*}
\vec{Z}^{0}_{\ell}=\left(
\begin{array}{c}\partial_{x}Q_{\ell}\\ -\ell\partial^{2}_{x}Q_{\ell}
\end{array}\right),\quad \vec{Z}^{\pm}_{\ell}=\left(
\begin{array}{c} \Upsilon_{\ell,1}^\pm\\ \Upsilon_{\ell,2}^\pm
\end{array}\right)
\end{equation*}
where
\[
\Upsilon_{\ell,1}^\pm = \bigg(\ell\partial_{x}Y_{\ell}\pm \frac{{\nu_{0}}}{\sqrt{1-\ell^{2}}}Y_{\ell}\bigg)e^{\pm \frac{\ell{\nu_{0}}}{\sqrt{1-\ell^{2}}}x}
\quad \mbox{and}\quad 
\Upsilon_{\ell,2}^\pm = Y_{\ell}e^{\pm \frac{ \ell{\nu_{0}}}{\sqrt{1-\ell^{2}}}x}.
\]
The functions $\vec{Z}_{\ell}^{\pm}$ are eigenfunctions of the linearized matrix operator $\mathcal{H}_{\ell}J$ and are explicitly determined by solving the eigenvalue problem $\mathcal{H}_{\ell}J\vec{Z}=\lambda \vec{Z}$.

\begin{lemma}[\cite{CMkg}] Let $\ell\in (-1,1)$.

\emph{(i) Decay properties of $\Upsilon_{\ell,1}^\pm$ and $\Upsilon_{\ell,2}^\pm$.}
It holds, on $\R$, for $k=1,2$,
\begin{equation}\label{est:decayUpsilon}
\bigg|\frac{\d ^\alpha}{\d x^{\alpha}}\Upsilon_{\ell,k}^\pm(x)\bigg|
\lesssim e^{-|x|}\quad \mbox{for any $\alpha\in \mathbb{N}$}.
\end{equation}
\emph{(ii) Properties of $\mathcal{H}_{\ell}$ and $\mathcal{H}_{\ell}J$.}
It holds
\begin{equation}\label{idenZ}
\mathcal{H}_{\ell}\vec{Z}^{0}_{\ell}=0,\quad \big(\vec{Z}^{0}_{\ell},\vec{Z}_{\ell}^{\pm}\big)_{L^{2}}=0
\quad \mbox{and}\quad \mathcal{H}_{\ell}J\big(\vec{Z}^{\pm}_{\ell}\big)=\mp {\nu_{0}}(1-\ell^{2})^{\frac{1}{2}}\vec{Z}_{\ell}^{\pm}.
\end{equation}

\emph{(iii) Coercivity property of $\mathcal{H}_{\ell}$.}
There exists $\nu>0$ such that, for all $\vec{v}\in H^{1}\times L^{2}$,
\begin{equation}\label{coer:H}
\big(\mathcal{H}_{\ell}\vec{v},\vec{v}\big)_{L^{2}}\ge \nu \|\vec{v}\|_{\E}^{2}-\nu^{-1}\big((\vec{v},\vec{Z}^{0}_{\ell})_{L^{2}}^{2}+(\vec{v},\vec{Z}_{\ell}^{+})_{L^{2}}^{2}+(\vec{v},\vec{Z}_{\ell}^{-})_{L^{2}}^{2}\big).
\end{equation}
\end{lemma} 
\begin{proof}
We prove (i) for $\alpha=0$; the estimates for $\alpha\geq 1$ are proved similarly.
By \eqref{est:decayY}, we have, for $k=1,2$, on $\R$,
\[
|\Upsilon_{\ell,k}^\pm(x)|\lesssim e^{- \frac{\sqrt{1+\nu_0^2}- |\ell| \nu_{0}} {\sqrt{1-\ell^{2}}}|x|}.
\]
But an elementary computation yields
\[
\left(\sqrt{1+\nu_0^2}- |\ell| \nu_{0}\right)^2
=\left(|\ell| \sqrt{1+\nu_0^2}- \nu_{0}\right)^2 + 1-\ell^2,
\]
and so $\sqrt{1+\nu_0^2}- |\ell| \nu_{0} \geq \sqrt{1-\ell^{2}}$, 
which implies $|\Upsilon_{\ell,k}^\pm(x)|\lesssim e^{-|x|}$.

For (ii) and (iii), we refer to Proposition 1 and the proofs of Lemma 2, Lemma 3 and Proposition 2 in~\cite{CMkg}.
\end{proof}

\subsection{Decomposition of the solution around multi-solitary waves}

We recall general results on solutions of~\eqref{NLKGvec} that are close to the sum of $N\ge 2$ decoupled solitary waves.
For any $n\in \{1,\ldots, N\}$, let $\sigma_{n}=\pm 1$ and $\ell_{n}\in (-1,1)$ satisfy
\[-1<\ell_{1}<\ell_{2}<\cdots<\ell_{N}<1.\]
Let
$t\mapsto y_{n}(t)\in \RR$ be $C^{1}$ functions such that 
\begin{equation}\label{dis:y}
y_{n+1}-y_{n}\gg 1\quad \mbox{for any $n=1,\ldots N-1$}.
\end{equation}
For $n\in \{1,\ldots,N\}$, define
\begin{equation*}
Q_{n}=\sigma_{n}Q_{\ell_{n}}(\cdot-y_{n}),\quad \vec{Q}_{n}=\left(
\begin{array}{c}
Q_{n}\\
-\ell_{n}\partial_{x}Q_{n}
\end{array}\right)
\end{equation*}
and
\begin{equation*}
\vec{Z}_{n}^{0}=\sigma_{n}\vec{Z}_{\ell_{n}}^{0}(\cdot-y_{n}),\quad \vec{Z}_{n}^{\pm}=\vec{Z}_{\ell_{n}}^{\pm}(\cdot-y_{n})=\left(\begin{array}{c}
\Upsilon_{n,1}^{\pm}\\
\Upsilon_{n,2}^{\pm}
\end{array}\right).
\end{equation*}

We recall a decomposition result for solutions of~\eqref{NLKGvec}.
\begin{lemma}\label{le:decom}
There exist $L_{0}>0$ and $0<\delta_{0}\ll 1$ such that 
if $\vec{u}=(u_{1},u_{2})$ is a solution of~\eqref{NLKGvec} on $[0,T_{0}]$, where $T_{0}>0$, such that for all $t\in [0,T_{0}]$
\begin{equation}\label{estQ0}
\inf_{z_{n+1}-z_{n}>L_{0}}\|\vec{u}(t)-\sum_{n=1}^{N}\sigma_{n}\vec{Q}_{\ell_{n}}(\cdot-z_{n})\|_{\E} < \delta_{0},
\end{equation}
then there exist $C^{1}$ functions $y_{1}(t),\ldots,y_{N}(t)$ on $[0,T_{0}]$ such that, $\vec{\varphi}$ being defined by
\begin{equation}\label{def:psi}
\vec{\varphi}=\left(
\begin{array}{c} \varphi_{1}\\ \varphi_{2}
\end{array}\right),\quad \vec{u}=\sum_{n=1}^{N}\vec{Q}_{n}+\vec{\varphi},
\end{equation}
it satisfies
\begin{equation}\label{phiorth}
\big(\vec{\varphi},\vec{Z}_{n}^{0}\big)_{L^{2}}=0,\quad \mbox{for $n=1,\ldots,N$,}
\end{equation}
and
\begin{equation}\label{init}
\|\vec{\varphi}\|_{\E}\lesssim \delta_{0},\quad y_{n+1}-y_{n}\ge \frac{3}{4}L_{0}\quad \mbox{for $n=1,\ldots,N-1$}.
\end{equation}
\end{lemma}
\begin{proof}
The proof of the decomposition lemma relies on a standard argument based on the Implicit function Theorem (see $\emph{e.g.}$ Lemma 5 and Appendix B in~\cite{CMkg})
and we omit it.
\end{proof}
Set
\begin{equation}\label{def:UG}
\vec{U}=\left(
\begin{array}{c} U_{1}\\ U_{2}
\end{array}\right)=\sum_{n=1}^{N}\vec{Q}_{n}
\quad \mbox{and}\quad 
G=f(U_{1})-\sum_{n=1}^{N}f(Q_{n}).
\end{equation}
\begin{lemma}[Equation of $\vec{\varphi}$]
The function $\vec{\varphi}$ satisfies
\begin{equation}\label{syst_e}\left\{\begin{aligned}
\partial_{t}\varphi_{1} & = \varphi_{2} + {\rm Mod}_{1},\\
\partial_{t}\varphi_{2} & 
= \px^{2} \vp-\vp +f\left(U_{1}+\varphi_{1}\right) 
- f\left(U_{1}\right) + G + {\rm Mod}_{2},
\end{aligned}\right.\end{equation} 
where
\begin{equation}\label{def:mod}
{\rm Mod}_{1}=\sum_{n=1}^{N}(\dot{y}_{n}-\ell_{n})\partial_{x}Q_{n},\quad 
{\rm Mod}_{2}=-\sum_{n=1}^{N}(\dot{y}_{n}-\ell_{n})\ell_{n}\partial_{x}^{2}Q_{n}.
\end{equation}
\end{lemma}
\begin{proof}
First, from the definition of $\vec{\varphi}=(\varphi_{1},\varphi_{2})$ in~\eqref{def:psi},
\begin{equation*}
\pt \vp=\pt u_{1}-\pt U_{1}=\vpp +U_{2}-\pt U_{1}=\vpp 
+\sum_{n=1}^{N}(\dot{y}_{n}-\ell_{n})\px Q_{n}.
\end{equation*}
Second, using~\eqref{NLKGvec},
\begin{equation*}
\pt \vpp=\pt u_{2}-\pt U_{2}=\px^{2}u_{1}-u_{1}+f(u_{1})-\sum_{n=1}^{N}\dot{y}_{n}\big(\ell_{n}\px^{2}Q_{n}\big).
\end{equation*}
We observe from~\eqref{def:psi} and $-(1-\ell^{2}_{n})\px^{2}Q_{n}+Q_{n}-f(Q_{n})=0$,
\begin{equation*}
\px^{2}u_{1}-u_{1}+f(u_{1})=\px ^{2}\vp-\vp +f(U_{1}+\vp)-\sum_{n=1}^{N}f(Q_{n})
+\sum_{n=1}^{N}\ell_{n}^{2}\px ^{2}Q_{n}.
\end{equation*}
Therefore, from the definition of $G$ and $\rm{Mod}_{2}$, we obtain~\eqref{syst_e}.
\end{proof}

For future reference, we present in Lemmas~\ref{le:point} and \ref{tech:le} some technical results.
First, we give several pointwise estimates concerning the nonlinearity $f$.
\begin{lemma}\label{le:point}
For any $(s_{n})_{n=1,\ldots,N}\in {\mathbb{R}}^{N}$, the following estimates hold.
\begin{enumerate}
\item For $1<q<p\le 3$, 
\begin{equation}\label{est:tay3}
\left|f'\left(\sum_{n=1}^{N}s_{n}\right)-\sum_{n=1}^{N}f'(s_{n})\right|\lesssim 
\sum_{n\ne n'}|s_{n}|^{\frac{q-1}{2}}|s_{n'}|^{\frac{q-1}{2}}+\sum_{n\ne n'}|s_{n}|^{\frac{p-1}{2}}|s_{n'}|^{\frac{p-1}{2}}.
\end{equation}
\item For $1<q\le 3<p<\infty$, 
\begin{equation}\label{est:tay4}
\left|f'\left(\sum_{n=1}^{N}s_{n}\right)-\sum_{n=1}^{N}f'(s_{n})\right|\lesssim 
\sum_{n\ne n'}|s_{n}|^{\frac{q-1}{2}}|s_{n'}|^{\frac{q-1}{2}}+\sum_{n\ne n'}|s_{n}|^{p-2}|s_{n'}|.
\end{equation}
\item For $3<q<p<\infty$, 
\begin{equation}\label{est:tay5}
\left|f'\left(\sum_{n=1}^{N}s_{n}\right)-\sum_{n=1}^{N}f'(s_{n})\right|\lesssim 
\sum_{n\ne n'}|s_{n}|^{q-2}|s_{n'}|+\sum_{n\ne n'}|s_{n}|^{p-2}|s_{n'}|.
\end{equation}
\item For $1<q<p<\infty$, 
\begin{equation}\label{est:tay2}
\left|f\left(\sum_{n=1}^{N}s_{n}\right)-\sum_{n=1}^{N}f(s_{n})\right|\lesssim \sum_{n\ne n'}|s_{n}|^{q-1}|s_{n'}|+\sum_{n\ne n'}|s_{n}|^{p-1}|s_{n'}|,
\end{equation}
\begin{equation}\label{est:tay1}
\left|F\left(\sum_{n=1}^{N}s_{n}\right)-\sum_{n=1}^{N}F(s_{n})\right|\lesssim \sum_{n\ne n'}|s_{n}|^{q}|s_{n'}|+\sum_{n\ne n'}|s_{n}|^{p}|s_{n'}|.
\end{equation}
\item For $1<q<p<\infty$,
\begin{equation}\label{est:tay6}
\left|f(s_1+s_2)-f(s_1)-f'(s_1)s_2\right|\lesssim |s_2|^2+|s_2|^q+|s_2|^p,
\end{equation}
\begin{equation}\label{est:tay7}
\left|F(s_1+s_2)-F(s_1)-f(s_1)s_2-\frac{1}{2}f'(s_1)s_2^2\right|
\lesssim |s_2|^3 + |s_2|^{q+1} + |s_2|^{p+1}.
\end{equation}
\end{enumerate}
\end{lemma}
\begin{proof}

Proof of (i)-(ii)-(iii).
Recall that $f'(s)=p|s|^{p-1}-q |s|^{q-1}$ for $p,q$ such that $1<q<p<\infty$.
Let $(s_{n})_{n=1,\ldots,N}\in \mathbb{R}^{N}$ and let $m\in \{1,\ldots,N\}$ be such that 
\begin{equation*}
|s_{m}|=\max_{n=1,\ldots,N}|s_{n}|.
\end{equation*}
By the triangle inequality,
\[
 \bigg| \Big|\sum_{n=1}^{N}s_{n}\Big|^{p-1}-\sum_{n=1}^N |s_{n}|^{p-1}\bigg|
 \lesssim 
\sum_{n\neq m} |s_n|^{p-1} + \bigg| \Big|\sum_{n=1}^{N}s_{n}\Big|^{p-1}-|s_{m}|^{p-1}\bigg|.
\]
First, by the definition of $m$, we observe that 
\[
\sum_{n\neq m} |s_n|^{p-1} \lesssim 
\begin{cases}
|s_{m}|^{\frac{p-1}{2}} \sum_{n\ne m}|s_{n}|^{\frac{p-1}{2}} & \mbox{if $1<p\le 3$}\\
|s_{m}|^{p-2} \sum_{n\ne m}|s_{n}| & \mbox{if $3<p$}.
\end{cases}
\]
Second, 
if $|s_{m}|\le 2 \sum\limits_{n\ne m}|s_{n}| $, then
\begin{equation*}
\bigg| \Big|\sum_{n=1}^{N}s_{n}\Big|^{p-1}-|s_{m}|^{p-1}\bigg|
\lesssim |s_{m}|^{p-1}
\lesssim 
\begin{cases}
|s_{m}|^{\frac{p-1}{2}} \sum_{n\ne m}|s_{n}|^{\frac{p-1}{2}} & \mbox{if $1<p\le 3$}\\
|s_{m}|^{p-2} \sum_{n\ne m}|s_{n}| & \mbox{if $3<p$}.
\end{cases}
\end{equation*}
Third, if $|s_{m}|> 2 \sum_{n\ne m}|s_{n}|$ then by the inequality $|(1+\mu)^{p-1} -1|\lesssim |\mu|$
for any $|\mu|\leq \frac 12$, we have
\begin{equation*}
 \bigg| \Big|\sum_{n=1}^{N}s_{n}\Big|^{p-1}-|s_{m}|^{p-1}\bigg|
=|s_{m}|^{p-1}\bigg|\bigg(1+\frac1 {s_{m}}{\sum_{n\ne m}s_{n}}\bigg)^{p-1}-1\bigg|
\lesssim|s_{m}|^{p-2} \sum_{n\ne m} |s_{n}|.
 \end{equation*}
Moreover, note that if $1\leq p\leq 3$ then $p-2\leq \frac{p-1}{2}$ and thus by the definition of $m$,
\[
|s_{m}|^{p-2} \sum_{n\ne m} |s_{n}| \lesssim |s_{m}|^{\frac{p-1}{2}} \sum_{n\neq m} |s_{n}|^{\frac{p-1}{2}}.
\]
Proceeding similarly for the term $|s|^{q-1}$ in $f'(s)$, we obtain (i), (ii) and (iii).

The proofs of (iv) and (v) follow from the same technique. 
\end{proof}
Second, using the pointwise estimates of Lemma~\ref{le:point}, we derive estimates related to the equation of $\vec{\varphi}$ and the nonlinear interaction term $G$ defined in \eqref{def:UG}.
Denote 
\begin{equation*}
q^{*}=\begin{cases}
 \frac{q-1}{2} & \mbox{for}\ 1<q\le 3\\
 q-2 & \mbox{for}\ 3<q<\infty
\end{cases}
 \quad\mbox{and}\quad
p^{*}= \begin{cases}
 \frac{p-1}{2} & \mbox{for}\ 1<p\le 3\\
 p-2 & \mbox{for}\ 3<p<\infty
\end{cases}
\end{equation*}
and observe that $0<q^*<p^*$.
Fix
\begin{equation}\label{def:gam0}
\gamma_{0}
=\frac{1}{8} \min\left(q^{*},1\right)>0\quad \mbox{and let}\ \ \theta=\sum_{n=1}^{N-1}e^{-4\gamma_{0}(y_{n+1}-y_{n})}.
\end{equation}
The introduction of $\gamma_0$ and $\theta$ will allow us to keep track of the exponent of 
exponential interactions between solitons (for instance, see the next lemma).
We will not keep track of this dependence in multiplicative constants.

\begin{lemma}\label{tech:le}
Assume~\eqref{dis:y}, we have
\begin{equation}\label{tech1}
\sum_{n\ne n'}\int_{\R}\left(\left|Q_{n}Q_{n'}\right|+\left|\px Q_{n}\px Q_{n'}\right|+\left|\px^{2} Q_{n}\px^{2} Q_{n'}\right|\right)\d x\lesssim \theta,
\end{equation}
\begin{equation}\label{tech4}
\sum_{n\ne n'}\int_{\RR}\left(\left|\px Q_{n}\Upsilon_{n',1}\right|+\left|\partial_{x}^{2}Q_{n}\Upsilon_{n',2}\right|\right)\d x\lesssim \theta,
\end{equation}
\begin{equation}\label{tech2}
\int_{\R}\bigg|F(U_{1})-\sum_{n=1}^{N}F(Q_{n})\bigg|\d x\lesssim \theta,
\end{equation}
\begin{equation}\label{tech3}
\|G\|_{L^{2}}+\sum_{n\ne n'}\|f'(Q_{n})\Upsilon_{n',2}\|_{L^{2}}+\|f'(U_{1})-\sum_{n=1}^{N}f'(Q_{n})\|_{L^{2}}\lesssim \theta.
\end{equation}
\end{lemma}
\begin{remark}
For convenience, all the terms above are bounded by the same error term $\sum_{n=1}^{N-1}e^{-4\gamma_0(y_{n+1}-y_{n})}$. Of course, in view of the definition of $\gamma_0$, these estimates are far from being optimal.
\end{remark}
\begin{proof}[Proof of Lemma~\ref{tech:le}]
\emph{Proof of~\eqref{tech1}}. First, by change of variable, for any $n\ne n'$,
\begin{equation*}
\int_{\RR}|Q_{n}Q_{n'}|\d x =H_{1}+H_{2},
\end{equation*}
where
\begin{align*}
&H_{1}=(1-\ell_{n}^{2})^{\frac{1}{2}}\int_{I_{1}}Q(x)Q\left(\frac{(y_{n}-y_{n'})+(1-\ell^{2}_{n})^{\frac{1}{2}}x}{(1-\ell^{2}_{n'})^{\frac{1}{2}}}\right)\d x,\\
&H_{2}=(1-\ell_{n}^{2})^{\frac{1}{2}}\int_{I_{2}}Q(x)Q\left(\frac{(y_{n}-y_{n'})+(1-\ell^{2}_{n})^{\frac{1}{2}}x}{(1-\ell^{2}_{n'})^{\frac{1}{2}}}\right)\d x,
\end{align*}
and
\begin{align*}
&I_{1}=\left\{x\in \RR: (1-\ell^{2}_{n})^{\frac{1}{2}} |x|<\frac{1}{2}|y_{n}-y_{n'}| \right\},\\
&I_{2}=\left\{x\in \RR: (1-\ell^{2}_{n})^{\frac{1}{2}} |x|\ge \frac{1}{2}|y_{n}-y_{n'}| \right\}.
\end{align*}
From the decay properties of $Q$ in~\eqref{Qdec} and then $\gamma_{0}\le \frac 18\le \frac 18(1-\ell_{n}^{2})^{-\frac{1}{2}}$, for any $x\in I_{2}$, we have 
\begin{equation*}
\left|Q\left(x\right)\right|\lesssim e^{-\frac12\frac{|y_{n}-y_{n'}|}{\sqrt{1-\ell_n^2}}} \lesssim e^{-4\gamma_0|y_{n}-y_{n'}|},
\end{equation*}
and so
\begin{equation*}
H_{2}\lesssim e^{-4\gamma_{0}|y_{n}-y_{n'}|}\int_{I_{2}}Q\left(\frac{(y_{n}-y_{n'})+(1-\ell^{2}_{n})^{\frac{1}{2}}x}{(1-\ell^{2}_{n'})^{\frac{1}{2}}}\right)\d x\lesssim e^{-4\gamma_{0}|y_{n}-y_{n'}|}\lesssim \theta.
\end{equation*}
Similarly, for any $x\in I_1$,
\begin{equation*}
\left|Q\left(\frac{(y_{n}-y_{n'})+(1-\ell^{2}_{n})^{\frac{1}{2}}x}{(1-\ell^{2}_{n'})^{\frac{1}{2}}}\right)\right|\lesssim e^{-4\gamma_0|y_{n}-y_{n'}|},
\end{equation*}
and so
\begin{equation*}
H_{1}\lesssim e^{-4\gamma_0|y_{n}-y_{n'}|}\int_{I_{1}}Q(x)\d x\lesssim e^{-4\gamma_{0}|y_{n}-y_{n'}|}\lesssim \theta.
\end{equation*}

This proves estimate~\eqref{tech1} for $Q_{n}Q_{n'}$. Using~\eqref{Qdec}, the rest of the proof of~\eqref{tech1} is the same.

\emph{Proof of~\eqref{tech4}}. From~\eqref{Qdec},~\eqref{est:decayUpsilon} and $(1-\ell_{n}^{2})^{-\frac{1}{2}}\ge 1$, we have 
\begin{equation*}
\sum_{n\ne n'}\left|\px Q_{n}\Upsilon_{n',1}\right|+\left|\partial_{x}^{2}Q_{n}\Upsilon_{n',2}\right|\lesssim
\sum_{n\ne n'}e^{-|x-y_{n}|}e^{-|x-y_{n'}|}.
\end{equation*}
Arguing as in the proof of~\eqref{tech1}, using $\gamma_0=\frac 18 \min(q^*,1)\le \frac{1}{8}$, we obtain~\eqref{tech4}.

\emph{Proof of~\eqref{tech2}}. From~\eqref{Qdec},~\eqref{est:tay1} and $1<q<p<\infty$, we infer
\begin{equation*}
\left|F(U_{1})-\sum_{n=1}^{N}F(Q_{n})\right|
\lesssim \sum_{n\ne n'}|Q_{n}|^{q}|Q_{n'}|+\sum_{n\ne n'}|Q_{n}|^{p}|Q_{n'}|
\lesssim \sum_{n\ne n'} |Q_{n}||Q_{n'}|.
\end{equation*}
Thus, estimate ~\eqref{tech2} follows from~\eqref{tech1}.

\emph{Proof of~\eqref{tech3}}. From~\eqref{est:tay2},
we have
\begin{equation*}
|G|\lesssim \sum_{n\ne n'}|Q_{n}|^{q-1}|Q_{n'}|+\sum_{n\ne n'}|Q_{n}|^{p-1}|Q_{n'}|.
\end{equation*}
By $1<q<p$,~\eqref{Qdec} and~\eqref{est:decayUpsilon},
\begin{equation*}
\sum_{n\ne n'}\left|f'(Q_{n})\Upsilon_{n',2}\right|\lesssim \sum_{n\ne n'}|Q_{n}|^{q-1}e^{-|x-y_{n'}|}.
\end{equation*}
Moreover, from~\eqref{est:tay3},~\eqref{est:tay4},~\eqref{est:tay5}, we infer,
for $1<q<p\leq 3$,
\begin{equation*}
\bigg|f'(U_{1})-\sum_{n=1}^{N}f'(Q_{n})\bigg|\lesssim
\sum_{n\ne n'}|Q_{n}|^{\frac{q-1}{2}}|Q_{n'}|^{\frac{q-1}{2}}+\sum_{n\ne n'}|Q_{n}|^{\frac{p-1}{2}}|Q_{n'}|^{\frac{p-1}{2}},
\end{equation*}
for $1<q\leq 3 <p$,
\begin{equation*}
\bigg|f'(U_{1})-\sum_{n=1}^{N}f'(Q_{n})\bigg|\lesssim
\sum_{n\ne n'}|Q_{n}|^{\frac{q-1}{2}}|Q_{n'}|^{\frac{q-1}{2}}+\sum_{n\ne n'}|Q_{n}|^{p-2}|Q_{n'}|,
\end{equation*}
and last, for $3<q<p$,
\begin{equation*}
\bigg|f'(U_{1})-\sum_{n=1}^{N}f'(Q_{n})\bigg|\lesssim
\sum_{n\ne n'}|Q_{n}|^{q-2}|Q_{n'}|+\sum_{n\ne n'}|Q_{n}|^{p-2}|Q_{n'}|.
\end{equation*}

Arguing as in the proof of~\eqref{tech1}, using $\gamma_0= \frac 18 \min(q^*,1)$, we obtain~\eqref{tech3}.
\end{proof}
Third, we derive the control of $\dot y_n$ from the orthogonality conditions~\eqref{phiorth}.
\begin{lemma}[Control of $\dot y_n$] \label{le:esty}
In the context of Lemma~\ref{le:decom}, we have 
\begin{equation}\label{est:y}
\sum_{n=1}^{N}\big|\dot{y}_{n}-\ell_{n}\big|\lesssim \|\vec{\varphi}\|_{\E}+\theta.
\end{equation}
\end{lemma}
\begin{proof}
First, we rewrite the equation of $\vec{\varphi}=(\vp,\vpp)$ as 
\begin{equation}\label{equ:vvp}
\pt \vec{\varphi}=\vec{\mathcal{L}}\vec{\varphi}+\vec{\rm{Mod}}+\vec{G}+\vec{R}_{1}+\vec{R}_{2},
\end{equation}
 where
 \begin{equation*}
 \vec{\mathcal{L}}=\left(\begin{array}{cc}0 & 1 \\ \partial_{x}^{2}-1+\sum_{n=1}^{N}f'(Q_{n}) & 0\end{array}\right),\quad 
 \vec{\rm{Mod}}=\left(\begin{array}{c} {\rm{Mod}_{1}} \\ {\rm{Mod}_{2}}\end{array}\right),
 \end{equation*}
 \begin{equation*}
 \vec{G}=\left(\begin{array}{c} 0 \\ G\end{array}\right),\quad \vec{R}_{1}=\left(\begin{array}{c}
 0\\ R_{1} \end{array}\right)
 =\left(\begin{array}{c}
 0\\ f(U_{1}+\vp)-f(U_{1})-f'(U_{1})\vp
 \end{array}\right),
 \end{equation*}
 and
 \begin{equation*}
 \vec{R}_{2}=\left(\begin{array}{c}
 0\\ R_{2} \end{array}\right)
 =\left(\begin{array}{c}
 0\\ \left(f'(U_{1})-\sum_{n=1}^{N}f'(Q_{n})\right)\vp
 \end{array}\right).
 \end{equation*}
 Second, from the orthogonality conditions~\eqref{phiorth},
 \begin{equation*}
 0=\frac{\rm{d}}{{\rm{d}} t}\left(\vec{\varphi},\vec{Z}_{n}^{0}\right)_{L^{2}}=\left(\pt \vec{\varphi},\vec{Z}_{n}^{0}\right)_{L^{2}}+\left(\vec{\varphi},\pt \vec{Z}_{n}^{0}\right)_{L^{2}}.
 \end{equation*}
 Thus, using~\eqref{equ:vvp},
 \begin{equation*}
 \begin{aligned}
 0=&\left(\vec{\mathcal{L}}\vec{\varphi},\vec{Z}_{n}^{0}\right)_{L^{2}}
 +\left(\vec{R}_{1},\vec{Z}_{n}^{0}\right)_{L^{2}}+\left(\vec{R}_{2},\vec{Z}_{n}^{0}\right)_{L^{2}}
 +\left(\vec{G},\vec{Z}_{n}^{0}\right)_{L^{2}}
 +\left(\vec{{\rm{Mod}}},\vec{Z}_{n}^{0}\right)_{L^{2}}\\
 &-(\dot{y}_{n}-\ell_{n})\left(\vec{\varphi},\px \vec{Z}_{n}^{0}\right)_{L^{2}}
 -\ell_{n}\left(\vec{\varphi},\px \vec{Z}_{n}^{0}\right)_{L^{2}}.
 \end{aligned}
 \end{equation*}
By integration by parts and the decay properties of $Q$ in~\eqref{Qdec}, the first term is 
 \begin{equation*}
 \left(\vec{\mathcal{L}}\vec{\varphi},\vec{Z}_{n}^{0}\right)_{L^{2}}
 =\left(\vec{\varphi},\vec{\mathcal{L}}^{t}\vec{Z}_{n}^{0}\right)_{L^{2}}=O(\|\vec{\varphi}\|_{\E}).
 \end{equation*}
 Next, by \eqref{est:tay6}, we infer
 \begin{equation*}
| R_{1}|\lesssim |\vp|^{2}+|\vp|^{q}+|\vp|^{p},
 \end{equation*}
 and by the Sobolev embedding theorem,
 \begin{equation*}
 \left|\left(\vec{R}_{1},\vec{Z}_{n}^{0}\right)_{L^{2}}\right|\lesssim \|\vec{\varphi}\|_{\E}^{2}+\|\vec{\varphi}\|_{\E}^{q}+\|\vec{\varphi}\|_{\E}^{p}.
 \end{equation*}
 Using the Cauchy-Schwarz inequality and~\eqref{tech3}, we obtain
 \begin{equation*}
\left|\left(\vec{R}_{2},\vec{Z}_{n}^{0}\right)_{L^{2}}\right|\lesssim \|\vec{\varphi}\|_{\E}\|f'(U_{1})-\sum_{n'=1}^{N}f'(Q_{n'})\|_{L^{2}}\lesssim \|\vec{\varphi}\|_{\E}^{2}+\theta^2.
 \end{equation*}
 Then, using~\eqref{tech3}, we have
 \begin{equation*}
 \big|\left(\vec{G},\vec{Z}_{n}^{0}\right)_{L^{2}}\big|\lesssim \|G\|_{L^{2}}
 \lesssim \theta.
 \end{equation*}
 Next, using the expression of $\vec{\rm{Mod}}$,
 \begin{equation*}
 \left(\vec{{\rm{Mod}}},\vec{Z}_{n}^{0}\right)_{L^{2}}
=(\dot{y}_{n}-\ell_{n})\left(\vec{Z}_{n}^{0},\vec{Z}_{n}^{0}\right)_{L^{2}}
 +\sum_{n'\ne n}(\dot{y}_{n'}-\ell_{n'})\left(\vec{Z}_{n'}^{0},\vec{Z}_{n}^{0}\right)_{L^{2}}.
 \end{equation*}
 Moreover, from~\eqref{tech1}, 
 \begin{equation*}
\sum_{n'\ne n}(\dot{y}_{n'}-\ell_{n'})\left(\vec{Z}_{n'}^{0},\vec{Z}_{n}^{0}\right)_{L^{2}}
=O\bigg( \theta \sum_{n'\ne n }|\dot{y}_{n'}-\ell_{n'}|\bigg).
 \end{equation*}
 Last, by the Cauchy-Schwarz inequality,
 \begin{equation*}
 \left|(\dot{y}_{n}-\ell_{n})\left(\vec{\varphi},\px \vec{Z}^{0}_{n}\right)_{L^{2}}\right|
 +\left|\ell_{n}\left(\vec{\varphi},\px \vec{Z}^{0}_{n}\right)_{L^{2}}\right|
 \lesssim \|\vec{\varphi}\|_{\E}\left(1+\left|\dot{y}_{n}-\ell_{n}\right|\right).
 \end{equation*}
 Gathering above estimates, we obtain
 \begin{equation*}
 \big|\dot{y}_{n}-\ell_{n}\big|
 \lesssim \|\vec{\varphi}\|_{\E}+\theta 
 +O\bigg( (\theta+\|\vec{\varphi}\|_{\E}) \sum_{n'=1}^{N}|\dot{y}_{n'}-\ell_{n'}|\bigg) .
 \end{equation*}
 Summing over $n$, we find
 \begin{equation*}
 \sum_{n=1}^{N}\big|\dot{y}_{n}-\ell_{n}\big|
 \lesssim \|\vec{\varphi}\|_{\E}+\theta 
 +O\bigg( (\theta+\|\vec{\varphi}\|_{\E}) \sum_{n'=1}^{N}|\dot{y}_{n'}-\ell_{n'}|\bigg),
 \end{equation*}
 which implies~\eqref{est:y} for $\delta_{0}$ small enough and $L_{0}$ large enough.
\end{proof}
Last, we consider the equation of the exponential directions.
\begin{lemma}[Exponential directions]\label{le:estda}
In the context of Lemma~\ref{le:decom}, let $a^{\pm}_{n}=\big(\vec{\varphi},\vec{Z}_{n}^{\pm}\big)_{L^{2}}$ for all $n=1,\ldots,N$. It holds
\begin{equation}\label{ode:z}
\bigg|\frac{\rm{d}}{\rm{d}t}a_{n}^{\pm} \mp \alpha_{n}a_{n}^{\pm}\bigg|\lesssim \|\vec{\varphi}\|_{\E}^{2}+\|\vec{\varphi}\|_{\E}^{q}+\theta,
\end{equation}
where $\alpha_{n}=\nu_{0}(1-\ell_{n}^{2})^{\frac{1}{2}}$.
\end{lemma}
\begin{proof}
First, using~\eqref{equ:vvp}, we compute,
\begin{align*}
\frac{\rm d}{{\rm d}t}a_{1}^{\pm}
&=\left(\pt \vec{\varphi},\vec{Z}_{1}^{\pm}\right)_{L^{2}}+\left( \vec{\varphi},\pt \vec{Z}_{1}^{\pm}\right)_{L^{2}}\\
&=\left(\vec{\mathcal{L}}\vec{\varphi},\vec{Z}_{1}^{\pm}\right)_{L^{2}}-\ell_{1}\left(\vec{\varphi},\px \vec{Z}_{1}^{\pm}\right)_{L^{2}}+\left(\vec{G},\vec{Z}_{1}^{\pm}\right)_{L^{2}}
\\&\quad +\left(\vec{R}_{1},\vec{Z}_{1}^{\pm}\right)_{L^{2}}
+\left(\vec{R}_{2},\vec{Z}_{1}^{\pm}\right)_{L^{2}}
+\left(\vec{\rm{Mod}},\vec{Z}_{1}^{\pm}\right)_{L^{2}}-(\dot{y}_{1}-\ell_{1})\left(\vec{\varphi},\px \vec{Z}_{1}^{\pm}\right)_{L^{2}}.
\end{align*}
Note that, from~\eqref{idenZ} and integration by parts,
\begin{align*}
\left(\vec{\mathcal{L}}\vec{\varphi},\vec{Z}_{1}^{\pm}\right)_{L^{2}}-\ell_{1}\left(\vec{\varphi},\px \vec{Z}_{1}^{\pm}\right)_{L^{2}}
&=-\left(\vec{\varphi},\left(\mathcal{H}_{\ell_{1}}J\vec{Z}_{\ell_1}^{\pm}\right)(\cdot-y_{1})\right)_{L^{2}}\\
&\quad +\sum_{n=2}^{N}\left(\vp,f'(Q_{n})\Upsilon_{1,2}^{\pm}\right)_{L^{2}}\\
&=\pm \alpha_{1}a_{1}^{\pm}+\sum_{n=2}^{N}\left(\vp,f'(Q_{n})\Upsilon_{1,2}^{\pm}\right)_{L^{2}}.
\end{align*}
Using~\eqref{tech3} and the Cauchy-Schwarz inequality,
\begin{align*}
&\bigg|\sum_{n=2}^{N}\left(\vp,f'(Q_{n})\Upsilon_{1,2}^{\pm}\right)_{L^{2}}\bigg|\lesssim \sum_{n=2}^{N}\|f'(Q_{n})\Upsilon_{1,2}^{\pm}\|_{L^{2}}\|\vp\|_{L^{2}}\lesssim
\|\vp\|_{L^{2}}^{2}+\theta^2.
\end{align*}
Next, using~\eqref{tech3}, \eqref{est:tay6} and the Sobolev embedding theorem
\begin{equation*}
\left|\left(\vec{R}_{1},\vec{Z}_{1}^{\pm}\right)_{L^{2}}\right|\lesssim
\int_{\RR}\left(|\vp|^{2}+|\vp|^{q}+|\vp|^{p}\right)\d x\lesssim \|\vec{\varphi}\|_{\E}^{2}+\|\vec{\varphi}\|_{\E}^{q}+\|\vec{\varphi}\|_{\E}^{p},
\end{equation*}
\begin{equation*}
\left|\left(\vec{R}_{2},\vec{Z}_{1}^{\pm}\right)_{L^{2}}\right|\lesssim \|f'(U_{1})-\sum_{n=1}^{N}f'(Q_{n})\|_{L^{2}}\|\vp\|_{L^{2}}\le \|\vec{\varphi}\|_{\E}^{2}+\theta^2,
\end{equation*}
and
\begin{equation*}
\left|\left(\vec{G},\vec{Z}_{1}^{\pm}\right)_{L^{2}}\right|\lesssim \|{G}\|_{L^{2}}
\lesssim \theta.
\end{equation*}
Moreover, from~\eqref{idenZ},~\eqref{tech4} and~\eqref{est:y},
\begin{align*}
\left(\vec{\rm{Mod}},\vec{Z}_{1}^{\pm}\right)_{L^{2}}
&=\sum_{n=2}^{N}(\dot{y}_{n}-\ell_{n})\left(\vec{Z}_{n}^{0},\vec{Z}_{1}^{\pm}\right)_{L^{2}}=O\left(\|\vec{\varphi}\|_{\E}^{2}+\theta^2\right).
\end{align*}
Last, from~\eqref{est:y},
\begin{equation*}
\big|(\dot{y}_{1}-\ell_{1}) (\vec{\varphi},\px \vec{Z}_{1}^{\pm})_{L^{2}}\big|\lesssim \|\vec{\varphi}\|_{\E}^{2}+\theta^2.
\end{equation*}
Gathering the above estimates and proceeding similarly for $a_{n}^{\pm}$ for $n=2,\ldots,N$, we obtain~\eqref{ode:z}.
\end{proof}

\section{Monotonicity property for the 1D Klein-Gordon equation}\label{Se:3}

\subsection{Bootstrap setting}
Fix 
\begin{equation}\label{def:g1}
\gamma_1=\frac{1}{8}\min\left(\ell_{2}-\ell_{1},\cdots,\ell_{N}-\ell_{N-1}\right)>0.
\end{equation}
We introduce the following bootstrap estimates: for $C_{0}\gg 1$ to be chosen later,
\begin{equation}\label{Bootset}\left\{
\begin{aligned}
& \|\vec{\varphi}(t)\|_{\E}\le C_{0}\left(\delta+e^{-\gamma_{0}L}\right),\\
& \min \{(y_{n+1}-y_{n})(t); \ n=1,\ldots,N\}\ge \frac{9}{10}L+2\gamma_{1}t,\\
&\sum_{n=1}^{N}|a_{n}^{+}(t)|^{2}\le C_{0}^{\frac{3}{2}}\left(\delta^{2}+e^{-2\gamma_{0}L}\right),\\
& \sum_{n=1}^{N}|a_{n}^{-}(t)|^{2}\le C_{0}^{\frac{7}{4}}\left(\delta^{2}+e^{-2\gamma_{0}L}\right).
\end{aligned}\right.
\end{equation}
Recall that $\gamma_0$ is defined in~\eqref{def:gam0}.
Let $\vec{u}_{0}$ satisfy~\eqref{estQ0} and~\eqref{Bootset} at $t=0$,
let $\vec{u}$ be the corresponding solution of~\eqref{NLKGvec} and define
\begin{equation}\label{defT*}
T_{*}(\vec{u}_{0})=\sup \{\mbox{$t\in [0,\infty): \vec{u}$ satisfies \eqref{estQ0} and~\eqref{Bootset} holds on $[0,t]$}\},
\end{equation}
Note that, from~\eqref{Bootset}, for all $t\in[0,T_{*}(\vec{u}_{0})]$,
\begin{equation}\label{est:theta}
\theta(t)=\sum_{n=1}^{N-1}e^{-4\gamma_{0}(y_{n+1}(t)-y_{n}(t))}
\lesssim e^{-\left(\frac{18}{5}L+8\gamma_{1} t\right)}\lesssim e^{-3\gamma_{0}L}.
\end{equation}

Note also that, for $C_{0}$ large enough, from~\eqref{est:y} and~\eqref{Bootset}, for all $n=1,\ldots,N$ and $t\in[0,T_{*}(\vec{u}_{0})]$,
\begin{equation*}
\left|\dot{y}_{n}-\ell_{n}\right|
\le C_{0}\bigg(\|\vec{\varphi}\|_{\E}+\theta\bigg)
\le 2C_{0}^{2}\left(\delta+e^{-\gamma_{0}L}\right)\le \gamma_{1} ,
\end{equation*}
for $L$ large enough and $\delta$ small enough (depending on $C_{0}$).
Integrating on $[0,t]$ for any $t\in[0,T_{*}(\vec{u}_{0})]$, we obtain,
for all $n=1,\ldots,N$,
\begin{equation}\label{est:yt}
\left|y_{n}(t)-\ell_{n}t-y_{n}(0)\right|\le \gamma_1 t.
\end{equation}

\subsection{Monotonicity property}
First, we fix suitable cut-off functions. Let $\chi$ be a $C^{3}$ function such that 
\begin{equation*}
\chi'\ge 0,\quad \sqrt{\chi}\in C^{1},\quad \sqrt{1-\chi}\in C^{1},\quad 
\chi(x)= 
\begin{cases}
0 & \mbox{for $x\le -1$},\\
1 & \mbox{for $x>1$}.
\end{cases}
\end{equation*}

Set
\begin{equation*}
  \beta_{n}=\frac{\ell_{n-1}+\ell_{n}}{2},\quad \bar{y}^{0}_{n}=\frac{y_{n-1}(0)+y_{n}(0)}{2},\quad 
  \chi_{n}(t,x)=\chi\left(\frac{x-\beta_{n} t-\bar{y}^{0}_{n}}{(t+a)^{\alpha}}\right),
\end{equation*}
for $n=2,\ldots,N$,
where $\alpha$ and $a$ are chosen so that 
\begin{equation*}
\frac{1}{2}<\alpha<\frac{4}{7}\quad \mbox{and}\quad a=\left(\frac{L}{10}\right)^{\frac{1}{\alpha}}.
\end{equation*}
We also set
\begin{equation*}
\beta_{1}=0,\quad \chi_{1}=1,\quad \chi_{N+1}=0.
\end{equation*}
Let
\begin{equation*}
\psi_{n}(t,x)=\chi_{n}(t,x)-\chi_{n+1}(t,x)\quad \mbox{for $n=1,\ldots,N$}.
\end{equation*}
Note that $\psi_{n}\equiv 1$ around the solitary wave $Q_n$, and $\psi_{n}\equiv 0$ around the solitary waves $Q_{n'}$ for $n'\ne n$. Moreover,
\begin{equation}\label{decompsi}
\sum_{n=1}^{N}\psi_{n}=1\quad \mbox{and}\quad \chi_{n}=\sum_{n'=n}^{N}\psi_{n'} \quad \mbox{for }\ n=1,\ldots,N.
\end{equation}
 Set 
\begin{equation*}
\Omega_{n}(t)=\left\{x\in \RR:|x-\beta_{n}t-\bar{y}_{n}^{0}|\le (t+a)^{\alpha}\right\}\quad \mbox{for}\ t\in[0,T_{*}(\vec{u}_{0})].
\end{equation*}
From the definition of $\chi_{n}$ and $\Omega_{n}$, we have the following estimates
($\textbf{1}_I$ denotes the characteristic function of the interval $I$)
\begin{equation}\label{pxchi}
|\px \chi_{n}| \lesssim \frac{1}{(t+a)^{\alpha}}\textbf{1}_{\Omega_{n}},
\end{equation}
\begin{equation}\label{est:px2chi}
 |\partial^{2}_{x}\chi_{n} |+ |\partial_{tx}\chi_{n} |\lesssim \frac{1}{(t+a)^{2\alpha}}\textbf{1}_{\Omega_{n}},\quad 
 |\px^{3}\chi_{n} |\lesssim \frac{1}{(t+a)^{3\alpha}}\textbf{1}_{\Omega_{n}}.
\end{equation}
We give technical estimates related to the cut-off functions.
\begin{lemma}\label{le:error}
For all $n=1,\ldots,N$ and all $t\in[0,T_{*}(\vec{u}_{0})]$, the following holds.
\begin{enumerate}
	\item \emph{Estimates related to $\chi_{n}$}. We have 
	\begin{equation}\label{est:n'len2}
	\sum_{n'=1}^{n-1}\int_{\RR}\left(\left(\px Q_{n'}\right)^{2}+Q_{n'}^{2}\right)\chi_{n}\d x\lesssim e^{-4\gamma_{0}L},
	\end{equation}
\begin{equation}\label{est:n'len}
\sum_{n'=1}^{n-1}\left(\|\chi_{n} Q_{n'}\|_{L^{2}}+\|\chi_{n}\px Q_{n'}\|_{L^{2}}\right)
\lesssim e^{-2\gamma_{0}L}.
\end{equation}
\item \emph{Estimates related to $(1-\chi_{n})$}. We have 
\begin{equation}\label{est:nlen'2}
\sum_{n'=n}^{N}\int_{\RR}\left(\left(\px Q_{n'}\right)^{2}+Q_{n'}^{2}\right)(1-\chi_{n})\d x\lesssim e^{-4\gamma_{0}L},
\end{equation}
\begin{equation}\label{est:nlen'}
\sum_{n'=n}^{N}\left(\|(\chi_{n}-1) Q_{n'}\|_{L^{2}}+\|(\chi_{n}-1)\px Q_{n'}\|_{L^{2}}\right)
\lesssim e^{-2\gamma_{0}L}.
\end{equation} 
\item \emph{Estimates related to $\Omega_{n}$.} We have 
\begin{equation}\label{est:1lenleN}
\sum_{n'=1}^{N} \left(\|{\bf{1}}_{\Omega_{n}} Q_{n'}\|_{L^{2}}
+ \|{\bf{1}}_{\Omega_{n}} \px Q_{n'}\|_{L^{2}}\right)\lesssim
e^{-2\left(\gamma_{0}L + \gamma_1 t\right)}.
\end{equation} 
\item \emph{Estimates related to $\psi_{n'}$.} We have 
\begin{equation}\label{est:f'nnen}
\sum_{n'\ne n}\|f'(Q_{n})\psi_{n'}\|_{L^{2}}\lesssim e^{-2\gamma_{0} L},
\end{equation}
\begin{equation}\label{est:Upsilonpsi}
\sum_{n'\ne n}\|\Upsilon_{n,1}\psi_{n'}\|_{L^{2}}+\sum_{n' \ne n}\|\Upsilon_{n,2}\psi_{n'}\|_{L^{2}}\lesssim e^{-2\gamma_{0} L},
\end{equation}
\begin{equation}\label{est:Qpsi}
\sum_{n'\ne n}\|\px Q_{n}\psi_{n'}\|_{L^{2}}+\sum_{n'\ne n}\|\partial_{x}^{2} Q_{n}\psi_{n'}\|_{L^{2}}\lesssim e^{-2\gamma_{0} L}.
\end{equation}
\end{enumerate}
\end{lemma}
\begin{proof}

Proof of~(i). We prove~\eqref{est:n'len2};~\eqref{est:n'len} is proved similarly. By the definition of $\chi_{n}$, for any $x\in \mathrm{supp}\, \chi_{n}$,
\[
x-\beta_n t -\bar y_n^0 \geq -(t+a)^{\alpha},
\]
and so using~\eqref{est:yt}, for any $1\le n'\le n-1$,
\begin{equation*}
x-y_{n'}(t) \ge (\beta_{n}-\ell_{n'})t- \gamma_1 t+(\bar{y}^{0}_{n}-y_{n'}(0))-(t+a)^{\alpha}.\end{equation*}
By the definition of $\gamma_{1}$ and~\eqref{Bootset} at $t=0$, we have 
$\beta_{n}-\ell_{n'}\geq \frac 12 (\ell_{n}-\ell_{n-1})\ge 4\gamma_1$,
\[\bar{y}^{0}_{n}-y_{n'}(0)=\frac{1}{2}\left(y_{n}(0)+y_{n-1}(0)\right)-y_{n'}(0) \geq \frac 12 (y_n(0)-y_{n-1}(0))\geq \frac{9L}{20},\]
and by the definition of $a^\alpha = \frac L{10}$, $0<\alpha<1$,
\[
(t+a)^{\alpha} 
\leq t^{\alpha}+a^{\alpha}\leq t^\alpha + \frac L{10}
\leq \gamma_1t + \frac {L}{5},
\]
for $L$ large enough.
Therefore, from $\gamma_{0}\le \frac{1}{8}$, for any $x\in \mathrm{supp}\ \chi_{n}$,
\begin{equation*}
 x-y_{n'}(t)\ge 4\gamma_1 t-\gamma_1 t+\frac{9L}{20}-\gamma_1 t-\frac{L}{5}\ge 2\gamma_1t+\frac{L}{4}\ge 2(\gamma_{0} L+\gamma_1 t),
\end{equation*}
for $L$ large enough.
Based on the decay properties of $Q$ in~\eqref{Qdec} and above estimate,
\begin{equation*}
\begin{aligned}
\sum_{n'=1}^{n-1}\int_{\RR}\left(\left(\px Q_{n'}\right)^{2}+Q_{n'}^{2}\right)\chi_{n}\d x
&\lesssim \sum_{n'=1}^{n-1}\int_{ x-y_{n'}(t)\ge 2(\gamma_{0} L+\gamma_1 t)}e^{-\frac{2|x-y_{n'}(t)|}{\sqrt{1-\ell^{2}_{n'}}}}\d x\\
&\lesssim \int_{x\ge 2(\gamma_{0} L+\gamma_1 t)}e^{-2x}\d x\lesssim e^{-4(\gamma_{0} L+\gamma_{1} t)},
\end{aligned}
\end{equation*}
which implies~\eqref{est:n'len2}.

The proof of (ii) is the same.

Proof of~(iii). 
Let any $1\le n\le N$ and $1\le n'\le N$.
For any $x\in \Omega_{n}$, by the triangle inequality, $\gamma_1=\frac{1}{8}\min_{n'\ne n}|\ell_{n}-\ell_{n'}|$,~\eqref{Bootset} at $t=0$ and~\eqref{est:yt},
\begin{align*}
|x-y_{n'}(t) |
&\ge |(\ell_{n'} -\beta_n) t + {y}_{n'}(0)-\bar{y}_{n}^{0}|
-|y_{n'}(t)-\ell_{n'}t-y_{n'}(0)| - |x-\beta_nt -\bar y_n^0|\\
&\ge
|\ell_{n'} -\beta_{n}| t
+\left|{y}_{n'}(0)-\bar{y}_{n}^{0}\right|-\gamma_1t-(t+a)^{\alpha}
\ge 3\gamma_1 t+\frac{9L}{20}-(t+a)^{\alpha}.
\end{align*}
Thus, using similar arguments as in the proof of~\eqref{est:n'len2}, we obtain~\eqref{est:1lenleN}.

Proof of~(iv).
 We prove~\eqref{est:f'nnen}; from~\eqref{Qdec} and~\eqref{est:decayUpsilon}, we see that~\eqref{est:Upsilonpsi} and~\eqref{est:Qpsi} are proved similarly. Let $1\le n\le N$ and $1\le n'\le n-1$. From the definition of $\psi_{n'}$, we have, for any $x\in {\rm{supp}}\ \psi_{n'}$,
\begin{equation*}
x-\beta_{n'+1}t-\bar{y}_{n'+1}^{0}\le (t+a)^{\alpha}.
\end{equation*}
Thus, by $\gamma_1= \frac{1}{8}\min_{n'\ne n}|\ell_{n}-\ell_{n'}|$,~\eqref{Bootset} at $t=0$,~\eqref{est:yt} and $a^{\alpha}=\frac{L}{10}$,
for any $x\in {\rm{supp}}\ \psi_{n'}$
\begin{equation*}
\begin{aligned}
x-y_{n}(t)&\le \left(\beta_{n'+1}-\ell_{n}\right)t+(\bar{y}_{n'+1}(0)-y_{n}(0))+\gamma_1 t+(t+a)^{\alpha}\\
&\le -\frac{1}{2}(\ell_{n}-\ell_{n-1})t-\frac{1}{2}(y_{n}(0)-y_{n-1}(0))+\gamma_1 t+t^{\alpha}+a^{\alpha}\\
&\le -3\gamma_1 t-\frac{9L}{20}+t^{\alpha}+\frac{L}{10}\le -3\gamma_1 t-\frac{9L}{20}+\gamma_{1} t+\frac{L}{5}\le -\frac{L}{4},
\end{aligned}
\end{equation*}
for $L$ large enough.
Let $1\le n\le N$ and $n'\ge n+1$. From the definition of $\psi_{n'}$, we have, for any $x\in {\rm{supp}}\ \psi_{n'}$,
\begin{equation*}
x-\beta_{n'}t-\bar{y}_{n'}^{0}\ge -(t+a)^{\alpha}.
\end{equation*}
Using again $\gamma_1\le \frac{1}{8}\min_{n'\ne n}|\ell_{n}-\ell_{n'}|$,~\eqref{Bootset} at $t=0$,~\eqref{est:yt} and $a^{\alpha}=\frac{L}{10}$, for any $x\in {\rm{supp}}\ \psi_{n'}$,
\begin{equation*}
\begin{aligned}
x-y_{n}(t)
&\ge (\beta_{n'}-\ell_{n})t+(\bar{y}_{n'}^{0}-y_{n}(0))-\gamma_1t-(t+a)^{\alpha}\\
&\ge \frac{1}{2}(\ell_{n+1}-\ell_{n})t+\frac{1}{2}(y_{n+1}(0)-y_{n}(0))-\gamma_1t-t^{\alpha}-a^{\alpha}\\
&\ge 3\gamma_1 t+\frac{9L}{20}-t^{\alpha}-\frac{L}{10}\ge 3\gamma_{1}t+\frac{9L}{20}-\gamma_{1}t-\frac{L}{5}\ge \frac{L}{4}.
\end{aligned}
\end{equation*}
Thus, using similar arguments as in the proof of~\eqref{est:n'len2} and $\frac{p-1}{4}\le 2\gamma_{0}$, we obtain~\eqref{est:f'nnen}.
\end{proof}
Second, let 
\begin{equation}\label{def:ck}
c_{1}=\ell_{1},\quad c_{2}=\frac{\ell_{2}-\ell_{1}}{1-\beta_{2}\ell_{2}},\quad c_{n}=\left(\frac{\ell_{n}-\ell_{n-1}}{1-\beta_{n}\ell_{n}}\right)\prod_{n'=2}^{n-1}\left(\frac{1-\beta_{n'}\ell_{n'-1}}{1-\beta_{n'}\ell_{n'}}\right),
\end{equation}
for $n=3,\ldots,N$. Denote
\begin{equation}\label{def:tc}
\tilde{c}_{1}=1\quad \mbox{and}\quad 
\tilde{c}_{n}=1+\sum_{n'=2}^{n}c_{n'}\beta_{n'}\quad \mbox{for}\ n=2,\ldots,N.
\end{equation}
Such specific choices are explained in Remark~\ref{re:cn}.
By direct computation, one proves the following lemma.
\begin{lemma}
For $n=2,\ldots,N$,
\begin{equation}\label{eq:sumck2}
\sum_{n'=1}^{n}c_{n'}=\tilde{c}_{n}\ell_{n}
\quad \mbox{and}\quad 
\tilde{c}_{n}=\prod_{n'=2}^{n}
\left(\frac{1-\beta_{n'}\ell_{n'-1}}{1-\beta_{n'}\ell_{n'}}\right).
\end{equation}
\end{lemma}
\begin{proof}
We prove~\eqref{eq:sumck2} by induction.

{\bf{Step 1}.} For $n=2$. By direct computation,
\begin{equation*}
\tilde{c}_{2}=1+c_{2}\beta_{2}=1+\frac{\beta_{2}\ell_{2}-\beta_{2}\ell_1}{1-\beta_{2}\ell_2}=\frac{1-\beta_{2}\ell_1}{1-\beta_{2}\ell_2},
\end{equation*}
\begin{equation*}
c_{1}+c_{2}=\ell_{1}+\frac{\ell_2-\ell_1}{1-\beta_{2}\ell_{2}}=\frac{\ell_{2}-\beta_{2}\ell_{1}\ell_{2}}{1-\beta_2\ell_2}=\tilde{c}_{2}\ell_{2},
\end{equation*}
which implies~\eqref{eq:sumck2} for $n=2$.

{\bf{Step 2}.} Assuming that~\eqref{eq:sumck2} is true for $n=k$ with $2\leq k\leq N-1$,
we prove that it is also true for $n=k+1$. From the definition of $c_{n}$ for $n\ge 3$, we obtain
\begin{align*}
\tilde{c}_{k+1}&=\prod_{n'=2}^{k}\left(\frac{1-\beta_{n'}\ell_{n'-1}}{1-\beta_{n'}\ell_{n'}}\right)
+\left(\frac{\beta_{k+1}(\ell_{k+1}-\ell_{k})}{1-\beta_{k+1}\ell_{k+1}}\right)\prod_{n'=2}^{k}\left(\frac{1-\beta_{n'}\ell_{n'-1}}{1-\beta_{n'}\ell_{n'}}\right)\\
&=\left(1+\frac{\beta_{k+1}(\ell_{k+1}-\ell_{k})}{1-\beta_{k+1}\ell_{k+1}}\right)\prod_{n'=2}^{k}\left(\frac{1-\beta_{n'}\ell_{n'-1}}{1-\beta_{n'}\ell_{n'}}\right)=\prod_{n'=2}^{k+1}\left(\frac{1-\beta_{n'}\ell_{n'-1}}{1-\beta_{n'}\ell_{n'}}\right),
\end{align*}
and
\begin{align*}
\sum_{n'=1}^{k+1}c_{n'}
&=\tilde{c}_{k}\ell_{k}+c_{k+1}\\
&=\ell_{k}\prod_{n'=2}^{k}\left(\frac{1-\beta_{n'}\ell_{n'-1}}{1-\beta_{n'}\ell_{n'}}\right)
+\left(\frac{\ell_{k+1}-\ell_{k}}{1-\beta_{k+1}\ell_{k+1}}\right)\prod_{n'=2}^{k}\left(\frac{1-\beta_{n'}\ell_{n'-1}}{1-\beta_{n'}\ell_{n'}}\right)\\
&=\ell_{k+1}\prod_{n'=2}^{k+1}\left(\frac{1-\beta_{n'}\ell_{n'-1}}{1-\beta_{n'}\ell_{n'}}\right)
=\tilde{c}_{k+1}\ell_{k+1},
\end{align*}
which means that~\eqref{eq:sumck2} is true for $n=k+1$. By the induction argument, \eqref{eq:sumck2} is proved for any $n=2,\ldots,N$.
\end{proof}
Third, we introduce the following quantities, for $n=2,\ldots,N$,
\begin{equation}\label{def:Jn}
\mathcal{J}_{n}(\vec{u})=\mathcal{I}_{{n}}(\vec{u})+\beta_{n}E_{n}(\vec{u})+\beta_{n}F_{n}(\vec{u})
\end{equation}
where
\begin{align}
\mathcal{I}_{n}(\vec{u})&=2\int_{\RR}\left(\chi_{n}\px u_{1}+\frac{1-\beta^{2}_n}{2}\big(\px \chi_{n}\big) u_{1}\right)u_{2}\d x,\label{def:In}\\
E_{n}(\vec{u})&=\int_{\RR}\big((\partial_{x}u_1)^{2}+u_1^{2}+u_2^{2}-2F(u_1)\big)\chi_{n}\d x,
\label{def:En}\\ 
F_{n}(\vec{u})&=-\frac{\alpha}{t+a}\int_{\RR}\left( x-\beta_{n}t-\bar{y}^{0}_{n}\right) u_{1}u_{2} (\px \chi_{n})\d x.\label{def:Fn}
\end{align}

Last, we set
\begin{equation}\label{def:comE}
\mathcal{E}(\vec{u})=E(\vec{u})+c_{1}\mathcal{I}(\vec{u})+\sum_{n=2}^{N}c_{n}\mathcal{J}_{n}(\vec{u})
\end{equation}
where, as defined in the Introduction,
\begin{equation*}
E(\vec{u})=\int_{\RR}\big((\partial_{x}u_1)^{2}+u_1^{2}+u_2^{2}-2F(u_1)\big)\d x,\quad 
\mathcal{I}(\vec{u})=2\int_{\RR}(\px u_{1})u_2\d x.
\end{equation*}
\begin{remark}\label{re:cn}
(i) The quantity $\mathcal{I}_{n}(\vec{u})$ is a standard localized version of the momentum for a wave-type equation, while $E_{n}(\vec{u})$ is a standard localized version of the energy.

(ii) The expression $E(\vec{u})+c_{1}\mathcal{I}(\vec{u})+\sum_{n=2}^{N}c_{n}\left(\mathcal{I}_{n}(\vec{u})+E_{n}(\vec{u})\right)$ should be the correct functional for stability of the sum of $N$ solitons.
Indeed, due to the specific choices of $c_{n}$ in~\eqref{def:ck},
locally around each solitary wave $Q_{n}$ this functional is analogous to the quantity $\tilde{c}_{n}\left(E(\vec{u})+2\ell_{n}\mathcal{I}(\vec{u})\right)$ (see Lemma~\ref{le:expanE}).
The functional $E(\vec{u})+2\ell\,\mathcal{I}(\vec{u})$ is a usual quantity to prove stability of a soliton $Q_{\ell}$ with speed $\ell\in(-1,1)$.

(iii) For technical reasons, we need to add the refined term $F_{n}(\vec{u})$ to
the definition of $\mathcal{J}_n$.
\end{remark}
By expanding $\vec{u}(t)=\sum_{n=1}^{N}\vec{Q}_{n}(t)+\vec{\varphi}(t)$, we obtain the following formula.
\begin{lemma}\label{le:expanE}
For all $t\in [0,T_{*}(\vec{u}_{0})]$, we have
\begin{equation}\label{expan:E}
\begin{aligned}
\mathcal{E}(\vec{u})=&\sum_{n=1}^{N}\tilde{c}_{n}(1-\ell^{2}_{n})^{\frac{1}{2}}E(\vec{Q})+\sum_{n=1}^{N}
\tilde{c}_{n}H_{n}(\vec{\varphi},\vec{\varphi})\\
&+O\left(\frac{\|\vec{\varphi}\|^{2}_{\E}}{L}+\|\vec{\varphi}\|_{\E}^{3}+\|\vec{\varphi}\|_{\E}^{q+1}+e^{-3\gamma_{0}L}\right),
\end{aligned}
\end{equation}
where
\begin{equation*}
H_{n}(\vec{\varphi},\vec{\varphi})
=\int_{\RR}\big((\px \varphi_{1})^{2}+\varphi_{1}^{2}+\varphi_{2}^{2}
+2\ell_{n}(\px\vp)\vpp-f'(Q_{n})\varphi_{1}^{2}\big)\psi_{n} \d x,
\end{equation*}
for $n=1,\ldots,N$.
\end{lemma}
\begin{remark}
Error terms are estimated by $e^{-3\gamma_0 L}$ in this section. Observe that some loss is needed with respect to Section 2: Lemmas~\ref{tech:le},~\ref{le:esty} and~\ref{le:estda}.
\end{remark}
\begin{proof}
{\bf {Step 1.}} Expansion of $E(\vec{u})$. We prove the following estimate
\begin{equation}\label{est:E}
\begin{aligned}
E(\vec{u})&=\sum_{n=1}^{N}(1-\ell^{2}_{n})^{\frac{1}{2}}E(\vec{Q})\\
&\quad +2\sum_{n=1}^{N}\int_{\RR}\left(\ell_{n}\px Q_{n}\right)\left(\ell_{n}\px Q_{n}+\ell_{n}\px \vp-\vpp\right)\d x\\
&\quad +\int_{\R}\left((\px \varphi_{1})^{2}+\varphi_{1}^{2}+\varphi_{2}^{2}-\sum_{n=1}^{N}f'(Q_{n})\varphi_{1}
^{2}\right)\d x\\
&\quad+O\left(\|\vec{\varphi}\|^{3}_{\E}+\|\vec{\varphi}\|_{\E}^{q+1}+e^{-3\gamma_{0}L}\right).
\end{aligned}
\end{equation} 
Recall that $\vec U = \sum_{n=1}^N \vec Q_{n}$.
First, using the decomposition~\eqref{def:psi}, the definition of $G$, the equation
$-(1-\ell_{n}^{2})\partial_{x}^{2}Q_{{n}}+Q_{{n}}-f(Q_{n})=0$ and integration by parts, we find
\begin{align*}
E(\vec{u})&=E(\vec{U})
+2\sum_{n=1}^{N}\int_{\R}(\ell_{n}\partial_{x}Q_{n})\big(\ell_{n}\px \vp-\varphi_{2}\big)\d x\\
&\quad +\int_{\RR}\big((\partial_{x}\varphi_{1})^{2}+\varphi_{1}^{2}+\varphi_{2}^{2}-\sum_{n=1}^{N}f'(Q_{n})\varphi^{2}_{1}\big)\d x\\
&\quad +\tilde{E}_{1}+\tilde{E}_{2}+\tilde{E}_{3},
\end{align*}
where
\begin{align*}
\tilde{E}_{1}&=-2\int_{\R}\big(F(U_{1}+\varphi_{1})-F(U_{1})-f(U_{1})\varphi_{1}-\frac{1}{2}f'(U_{1})\varphi_{1}^{2}\big)\d x , \\
\tilde{E}_{2}&=-2\int_{\R}G\varphi_{1}\d x,\\
 \tilde{E}_{3}&=-\int_{\R}\big(f'(U_{1})-\sum_{n=1}^{N}f'(Q_{n})\big)\varphi^{2}_{1} \d x.
\end{align*}
By \eqref{est:tay7}, $1<q<p<\infty$
and the Sobolev embedding Theorem,
\begin{equation*}
\left|\tilde{E}_{1}\right|\lesssim \int_{\R}\left(|\vp|^{p+1}+|\varphi_{1}|^{q+1}+|\varphi_{1}|^{3}\right)\d x\lesssim \|\varphi_{1}\|_{H^{1}}^{3}+\|\varphi_{1}\|_{H^{1}}^{q+1}.
\end{equation*}
By~\eqref{tech3},~\eqref{est:theta}, the Cauchy-Schwarz inequality and the Young's inequality, we have
\begin{equation*}
\left|\tilde{E}_{2}\right|\lesssim \|G\|_{L^{2}}\|\vp\|_{L^{2}}\lesssim \|\vp\|_{L^{2}}^{3}+\theta^{\frac{3}{2}}\lesssim 
\|\vp\|_{L^{2}}^{3}+e^{-\frac{9}{2}\gamma_{0} L},
\end{equation*}
and
\begin{align*}
\left|\tilde{E}_{3}\right|
&\lesssim
\big\|f'(U_{1})-\sum_{n=1}^{N}f'(Q_{n})\big\|_{L^{2}}\|\varphi_{1}\|^{2}_{H^{1}}\lesssim\|\varphi_{1}\|^{3}_{H^{1}}+\theta^{3}\lesssim\|\vec{\varphi}\|_{\E}^{3}+e^{-9\gamma_{0}L}.
\end{align*}

Second, by direct computation,
\begin{align*}
E(\vec{U})
&=\sum_{n=1}^{N}(1-\ell_{n}^{2})^{\frac{1}{2}}E(\vec{Q})+2\sum_{n=1}^{N}\int_{\R}\big(\ell_{n}\partial_{x}Q_{n}\big)^{2}\d x\\
&\quad +\sum_{n\ne n'}\int_{\R}\big[(1+\ell_{n}\ell_{n'})(\partial_{x}Q_{n})(\partial_{x}Q_{n'})
+Q_{n}Q_{n'}\big]\d x\\
&\quad -2\int_{\R}\big(F(U_{1})-\sum_{n=1}^{N}F(Q_{n})\big)\d x.
\end{align*}
Moreover, using~\eqref{tech1},~\eqref{tech2} and~\eqref{est:theta},
\begin{equation*}
\sum_{n\ne n'}\int_{\RR}\big(|\px Q_{n}\px Q_{n'}|+|Q_{n}Q_{n'}|\big)\d x
+\int_{\R}\big|F(U_{1})-\sum_{n=1}^{N}F(Q_{n})\big|\d x\lesssim \theta\lesssim e^{-3\gamma_{0}L}.
\end{equation*}
We see that~\eqref{est:E} follows from above estimates.

{\bf {Step 2}.} Expansion of $\mathcal{I}(\vec{u})$. We claim 
\begin{equation}\label{est:I}
\begin{aligned}
\mathcal{I}(\vec{u})
=&-2\sum_{n=1}^{N}\int_{\RR}\left(\px Q_{n}\right)\left(\ell_{n}\px Q_{n}+\ell_{n}\px \varphi_{1}-\varphi_{2}\right)\d x\\
&+2\int_{\RR}\left(\px \varphi_{1}\right)\varphi_{2}\d x+O\left(e^{-3\gamma_{0}L}\right).
\end{aligned}
\end{equation}
By direct computation and~\eqref{def:psi},
\begin{align*}
\mathcal{I}(\vec{u})
=&-2\sum_{n=1}^{N}\int_{\RR}\left(\px Q_{n}\right)\left(\ell_{n}\px Q_{n}+\ell_{n}\px \varphi_{1}-\varphi_{2}\right)\d x\\
&+2\int_{\RR}\left(\px \varphi_{1}\right)\varphi_{2}\d x
-2\sum_{n\ne n'}\ell_{n'}\int_{\RR}\left(\px Q_{n}\right)\left(\px Q_{n'}\right)\d x.
\end{align*}
Thus, from~\eqref{tech1} and~\eqref{est:theta}, we obtain~\eqref{est:I}.

{\bf Step 3}. Expansion of $\mathcal{I}_{n}(\vec{u})$. We claim
\begin{equation}\label{est:In}
\begin{aligned}
\mathcal{I}_{n}(\vec{u})=&-2\sum_{n'=n}^{N}\int_{\RR}\left(\px Q_{n'}\right)
\left(\ell_{n'}\px Q_{n'}+\ell_{n'}\px \varphi_{1}-\varphi_{2}\right)\d x\\
&+2\int_{\RR}\chi_{n}(\px \varphi_{1})\varphi_{2}\d x+O\left(\frac{\|\vec{\varphi}\|_{\E}^{2}}{L}+\|\vec{\varphi}\|_{\E}^{3}+e^{-3\gamma_{0}L}\right).
\end{aligned}
\end{equation}
We decompose
\begin{equation*}
\mathcal{I}_{n}(\vec{u})=\mathcal{I}_{n}^{1}(\vec{u})+(1-\beta_{n}^{2})\mathcal{I}_{n}^{2}(\vec{u}),
\end{equation*}
where
\begin{equation*}
\mathcal{I}_{n}^{1}(\vec{u})=2\int_{\RR}\left(\chi_{n}\px u_{1}\right)u_{2}\d x,\quad 
\mathcal{I}_{n}^{2}(\vec{u})=\int_{\RR}\left(\px \chi_{n}\right)u_{1}u_{2}\d x.
\end{equation*}

\emph{Estimate on $\mathcal{I}_{n}^{1}$.}
We claim
\begin{equation}\label{est:I1}
\begin{aligned}
\mathcal{I}_{n}^{1}(\vec{u})&=-2\sum_{n'=n}^{N}\int_{\RR}\left(\px Q_{n'}\right)
\left(\ell_{n'}\px Q_{n'}+\ell_{n'}\px \varphi_{1}-\varphi_{2}\right)\d x\\
&\quad +2\int_{\RR}\chi_{n}(\px \varphi_{1})\varphi_{2}\d x+O\left(\|\vec{\varphi}\|_{\E}^{3}+e^{-3\gamma_{0}L}\right).
\end{aligned}
\end{equation}
By direct computation and~\eqref{def:psi},
\begin{align*}
\mathcal{I}_{n}^{1}(\vec{u})
&=-2\sum_{n'=n}^{N}\int_{\RR}\left(\px Q_{n'}\right)
\left(\ell_{n'}\px \varphi_{1}+\ell_{n'}\px Q_{n'}-\varphi_{2}\right)\d x\\
&\quad +2\int_{\RR}\chi_{n}(\px \varphi_{1})\varphi_{2}\d x
+\mathcal{I}_{n}^{1,1}+\mathcal{I}_{n}^{1,2}
+\mathcal{I}_{n}^{1,3},
\end{align*}
where
\begin{align*}
&\mathcal{I}_{n}^{1,1}=-2\sum_{n'\ne n''}\ell_{n'}
\int_{\RR}\chi_{n}\left(\px Q_{n'}\right)\left(\px Q_{n''}\right)\d x,\\
&\mathcal{I}_{n}^{1,2}=-2\sum_{n'=1}^{n-1}\int_{\RR}\left(\chi_{n}\px Q_{n'}\right)
\left(\ell_{n'}\px \varphi_{1}+\ell_{n'}\px Q_{n'}-\varphi_{2}\right)\d x,\\
&\mathcal{I}_{n}^{1,3}=-2\sum_{n'=n}^{N}\int_{\RR}\left(\chi_{n}-1\right)\left(\px Q_{n'}\right)
\left(\ell_{n'}\px \varphi_{1}+\ell_{n'}\px Q_{n'}-\varphi_{2}\right)\d x.
\end{align*}
From~\eqref{tech1} and~\eqref{est:theta}, 
\begin{equation*}
\big|\mathcal{I}_{n}^{1,1}\big|\lesssim \sum_{n'\ne n''}\int_{\RR}|\px Q_{n'}| |\px Q_{n''}|\d x\lesssim \theta\lesssim e^{-3\gamma_{0}L}.
\end{equation*} 
By~\eqref{est:n'len2},~\eqref{est:n'len},~~\eqref{est:nlen'2},~\eqref{est:nlen'}, the Cauchy-Schwarz inequality and the Young's inequality, we have 
\begin{align*}
\left|\mathcal{I}_{n}^{1,2}\right|
&\lesssim \sum_{n'=1}^{n-1}\left(\|\vec{\varphi}\|_{\E}\|\chi_{n}\px Q_{n'}\|_{L^{2}}\right)+\sum_{n'=1}^{n-1}\int_{\RR}(\px Q_{n'})^{2}\chi_{n} \d x\\
&\lesssim \|\vec{\varphi}\|_{\E}^{3}+\sum_{n'=1}^{n-1}\|\chi_{n}\px Q_{n'}\|^{\frac{3}{2}}_{L^{2}}+\sum_{n'=1}^{n-1}\int_{\RR}(\px Q_{n'})^{2}\chi_{n} \d x\lesssim \|\vec{\varphi}\|_{\E}^{3}
+e^{-3\gamma_{0}L},
\end{align*}
and
\begin{align*}
\left|\mathcal{I}_{n}^{1,3}\right|
&\lesssim \sum_{n'=n}^{N}\left(\|\vec{\varphi}\|_{\E}\|(\chi_{n}-1)\px Q_{n'}\|_{L^{2}}\right)
+\sum_{n'=n}^{N}\int_{\RR}(\px Q_{n'})^{2}(1-\chi_{n})\d x\\
&\lesssim \|\vec{\varphi}\|_{\E}^{3}+\sum_{n'=n}^{N}\left(\|(\chi_{n}-1)\px Q_{n'}\|^{\frac{3}{2}}_{L^{2}}+\int_{\RR}(\px Q_{n'})^{2}(1-\chi_{n})\d x\right)\lesssim \|\vec{\varphi}\|_{\E}^{3}
+e^{-3\gamma_{0}L}.
\end{align*}
We see that~\eqref{est:I1} follows from above estimates.

\emph{Estimate on $\mathcal{I}_{n}^{2}$.}
We decompose
\begin{equation*}
\mathcal{I}_{n}^{2}(\vec{u})=\mathcal{I}_{n}^{2,1}+\mathcal{I}_{n}^{2,2}+\mathcal{I}_{n}^{2,3},
\end{equation*}
where
\begin{align*}
\mathcal{I}_{n}^{2,1}
&=-\sum_{n',n''=1}^{N}\ell_{n'}\int_{\RR}(\px \chi_{n})(\px Q_{n'})Q_{n''}\d x,\\
\mathcal{I}_{n}^{2,2}&=\sum_{n'=1}^{N}\int_{\RR}(\px \chi_{n})\left(-\ell_{n'}(\px Q_{n'})\vp+Q_{n'}\vpp\right)\d x,\\
 \mathcal{I}_{n}^{2,3}&=\int_{\RR}(\px \chi_{n})\vp\vpp\d x.
\end{align*}
By~\eqref{est:1lenleN}, the Cauchy-Schwarz inequality and the Young's inequality,
\begin{align*}
\left|\mathcal{I}_{n}^{2,1}\right|\lesssim \sum_{n'=1}^{N}\int_{\RR}\left(|\px Q_{n'}|^{2}+|Q_{n'}|^{2}\right)\textbf{1}_{\Omega_{n}}\d x\lesssim e^{-4\gamma_{0} L},
\end{align*}
and
\begin{align*}
\left|\mathcal{I}_{n}^{2,2}\right|
&\lesssim \sum_{n'=1}^{N}\|\vec{\varphi}\|_{\E}\left(\|\px Q_{n'}\textbf{1}_{\Omega_{n}}\|_{L^{2}}+\|Q_{n'}\textbf{1}_{\Omega_{n}}\|_{L^{2}}\right)\\ 
&\lesssim \|\vec{\varphi}\|_{\E}^{3}+\sum_{n'=1}^{N}\left(\|\px Q_{n'}\textbf{1}_{\Omega_{n}}\|^{\frac{3}{2}}_{L^{2}}+\|Q_{n'}\textbf{1}_{\Omega_{n}}\|^{\frac{3}{2}}_{L^{2}}\right)
\lesssim \|\vec{\varphi}\|_{\E}^{3}+e^{-3\gamma_{0} L}.
\end{align*}
Moreover, from~\eqref{pxchi} and $a^{\alpha}=\frac{L}{10}$,
\begin{equation*}
\left|\mathcal{I}_{n}^{2,3}\right|\lesssim \frac{1}{(t+a)^{\alpha}}\int_{\RR}|\vp||\vpp|\textbf{1}_{\Omega_{n}}\d x\lesssim \frac{\|\vec{\varphi}\|_{\E}^{2}}{(t+a)^{\alpha}}\lesssim \frac{\|\vec{\varphi}\|_{\E}^{2}}{L}.
\end{equation*}
From above estimates, we conclude that
\begin{equation}\label{est:In2}
\begin{aligned}
\left|\mathcal{I}_{n}^{2}\right|\lesssim \frac{\|\vec{\varphi}\|_{\E}^{2}}{L}+ \|\vec{\varphi}\|_{\E}^{3}+e^{-3\gamma_{0}L}.
\end{aligned}
\end{equation}
We see that~\eqref{est:In} follows from~\eqref{est:I1} and~\eqref{est:In2}.

{\bf{Step 4.}} Expansion of $E_{n}(\vec{u})$. We claim
\begin{equation}\label{est:En}
\begin{aligned}
E_{n}(\vec{u})
&=\sum_{n'=n}^{N}(1-\ell^{2}_{n'})^{\frac{1}{2}}E(\vec{Q})
\\&\quad+2\sum_{n'=n}^{N}\int_{\RR}\left(\ell_{n'}\px Q_{n'}\right)\left(\ell_{n'}\px Q_{n'}+\ell_{n'}\px \varphi_{1}-\varphi_{2}\right)\d x\\
&\quad +\int_{\RR}\left((\px \varphi_{1})^{2}+\varphi_{1}^{2}+\varphi_{2}^{2}-\sum_{n'=n}^{N}f'(Q_{n'})\varphi_{1}^{2}\right)\chi_{n}\d x\\
&\quad +O\left(\|\vec{\varphi}\|_{\E}^{3}+\|\vec{\varphi}\|_{\E}^{q+1}
+e^{-3\gamma_{0}L}\right).
\end{aligned}
\end{equation}
First, from~\eqref{def:psi}, $-(1-\ell_{n'}^{2})\partial_{x}^{2}Q_{n'}+Q_{n'}-f(Q_{n'})=0$, integration by parts and an elementary computation,
\begin{align*}
E_{n}(\vec{u})=
&E_{n}(\vec{U})+2\sum_{n'=n}^{N}\int_{\RR}
\left(\ell_{n'}\px Q_{n'}\right)\left(\ell_{n'}\px \varphi_{1}-\varphi_{2}\right)\d x\\
&+\int_{\RR}\bigg(\left(\px \varphi_{1}\right)^{2}+\varphi_{1}^{2}+\varphi_{2}^{2}-\sum_{n'=n}^{N}f'(Q_{n'})\varphi_{1}^{2}\bigg)\chi_{n}\d x\\
&+\tilde{E}_{n}^{1}+\tilde{E}_{n}^{2}+\tilde{E}_{n}^{3}+\tilde{E}_{n}^{4}+\tilde{E}_{n}^{5}+\tilde{E}_{n}^{6},
\end{align*}
where
\begin{align*}
\tilde{E}_{n}^{1}&=-2\int_{\RR}
\bigg(F(U_{1}+\varphi_{1})-F(U_{1})-f(U_{1})\varphi_{1}-\frac{1}{2}f'(U_{1})\varphi_{1}^{2}\bigg)\chi_{n}\d x,
\\
\tilde{E}_{n}^{2}&=-2\int_{\RR}G\varphi_{1}\chi_{n}\d x,\\ 
\tilde{E}_{n}^{3}&=-\int_{\RR} \bigg(f'(U_{1})-\sum_{n'=n}^{N}f'(Q_{n'})\bigg)\varphi_{1}^{2}\chi_{n}\d x,
\end{align*}
and
\begin{align*}
 \tilde{E}_{n}^{4}&=-2\sum_{n'=1}^{N}(1-\ell^{2}_{n'})\int_{\RR}\big(\px Q_{n'}\big)\big(\px \chi_{n}\big)\varphi_{1}\d x,
\\
\tilde{E}_{n}^{5}&=2\sum_{n'=1}^{n-1}\int_{\RR}
\big(\ell_{n'}\px Q_{n'}\big)\big(\ell_{n'}\px \varphi_{1}-\varphi_{2}\big)\chi_{n}\d x,
\\
\tilde{E}_{n}^{6}&=2\sum_{n'=n}^{N}\int_{\RR}\big(\ell_{n'}\px Q_{n'}\big)\big(\ell_{n'}\px \vp-\vpp\big)(\chi_{n}-1)\d x.
\end{align*}
Using again~\eqref{est:tay7}, $1<q<p<\infty$ and the Sobolev embedding Theorem, we obtain
\begin{equation*}
\big|\tilde{E}_{n}^{1}\big|
\lesssim \int_{\RR}\left(|\vp|^{3}+|\vp|^{q+1}+|\vp|^{p+1}\right)\d x\lesssim \|\vec{\varphi}\|_{\E}^{3}+\|\vec{\varphi}\|_{\E}^{q+1}.
\end{equation*}
By~\eqref{tech3},~\eqref{est:theta}, the Cauchy-Schwarz inequality and the Young's inequality, 
\begin{equation*}
\left|\tilde{E}_{2}\right|\lesssim \|G\|_{L^{2}}\|\vp\|_{L^{2}}\lesssim \|\vp\|_{L^{2}}^{3}+\theta^{\frac{3}{2}}\lesssim 
\|\vp\|_{L^{2}}^{3}+e^{-\frac{9}{2}\gamma_{0} L},
\end{equation*}
Recall that, for any $x\in {\rm{supp}}\chi_{n}$ and $1\le n'\le n-1$,
\begin{align*}
|x-y_{n'}(t)|\ge \frac{1}{4}(\ell_{n}-\ell_{n-1})t+\frac{L}{4}\ge \frac{L}{4},
\end{align*}
for $L$ large enough. Thus, from $1<q<p<\infty$ and $\gamma_{0}\le \frac{p-1}{8}$, we have 
\begin{equation*}
\sum_{n'=1}^{n-1}\|f'(Q_{n'})\chi_{n}\|_{L^{\infty}}\lesssim e^{-\frac{1}{4}(q-1)L}\lesssim e^{-2\gamma_0L}.
\end{equation*}
Therefore, using again~\eqref{tech3},~\eqref{est:theta}, the Sobolev embedding Theorem and the Young's inequality,
\begin{align*}
 |\tilde{E}_{n}^{3} |&\lesssim \int_{\RR}\bigg(\big|f'(U_{1})-\sum_{n'=1}^{N}f'(Q_{n'})\big|+\sum_{n'=1}^{n-1}|f'(Q_{n'})|\bigg)\vp^{2}|\chi_{n}|\d x\\
 &\lesssim \|\vec{\varphi}\|_{\E}^{2}\left(\|f'(U_{1})-\sum_{n'=1}^{N}f'(Q_{n'})\|_{L^{2}}
 +\sum_{n'=1}^{n-1}\|f'(Q_{n'})\chi_{n}\|_{L^{\infty}}\right)\\
&\lesssim \|\vec{\varphi}\|_{\E}^{3}+\theta^{3}+\sum_{n'=1}^{n-1}\|f'(Q_{n'})\chi_{n}\|_{L^{\infty}}^{3}\lesssim \|\vec{\varphi}\|_{\E}^{3}+e^{-6\gamma_0L}.
\end{align*}
Next, by~\eqref{est:n'len},~\eqref{est:nlen'},~\eqref{est:1lenleN}, the Cauchy-Schwarz inequality and the Young's inequality, we have
\begin{align*}
 |\tilde{E}_{n}^{4} |&\lesssim \sum_{n'=1}^{N}\|\vec{\varphi}\|_{\E}\|\px Q_{n'}\textbf{1}_{\Omega_{n}}\|_{L^{2}}\lesssim\|\vec{\varphi}\|_{\E}^{3}+\sum_{n'=1}^{N}\|\px Q_{n'}\textbf{1}_{\Omega_{n}}\|_{L^{2}}^{\frac{3}{2}}
\lesssim \|\vec{\varphi}\|_{\E}^{3}+e^{-3\gamma_{0}L},
\\
 |\tilde{E}_{n}^{5} |&\lesssim \sum_{n'=1}^{n-1}\|\vec{\varphi}\|_{\E}\|\chi_{n}\px Q_{n'}\|_{L^{2}}\lesssim\|\vec{\varphi}\|^{3}_{\E}+\sum_{n'=1}^{n-1}\|\chi_{n}\px Q_{n'}\|_{L^{2}}^{\frac{3}{2}}\lesssim \|\vec{\varphi}\|_{\E}^{3}+e^{-3\gamma_{0}L},
\\
 |\tilde{E}_{n}^{6} |&\lesssim \sum_{n'=n}^{N}\|\vec{\varphi}\|_{\E}\|(\chi_{n}-1)\px Q_{n'}\|_{L^{2}}\lesssim\|\vec{\varphi}\|^{3}_{\E}+\sum_{n'=n}^{N}\|(\chi_{n}-1)\px Q_{n'}\|_{L^{2}}^{\frac{3}{2}}\lesssim \|\vec{\varphi}\|_{\E}^{3}+e^{-3\gamma_{0}L}.
\end{align*}

Second, by direct computation, 
\begin{align*}
E_{n}(\vec{U})
=&\sum_{n'=n}^{N}(1-\ell^{2}_{n'})^{\frac{1}{2}}E(\vec{Q})+2\sum_{n'=n}^{N}\int_{\RR}\left(\ell_{n'}\px Q_{n'}\right)^{2}\d x\\
&+\tilde{E}_{n}^{7}+\tilde{E}_{n}^{8}+\tilde{E}_{n}^{9}+\tilde{E}_{n}^{10},
\end{align*}
where
\begin{align*}
\tilde{E}_{n}^{7}
&=-2\int_{\RR}\bigg(F(U_{1})-\sum_{n'=1}^{N}F(Q_{n'})\bigg)\chi_{n}\d x,
\\
\tilde{E}_{n}^{8}
&=\sum_{n'=1}^{n-1}\int_{\RR}\left((1+\ell_{n'}^{2})(\px Q_{n'})^{2}+Q_{n'}^{2}-2F(Q_{n'})\right)\chi_{n}\d x,
\end{align*}
and
\begin{align*}
\tilde{E}_{n}^{9}
&=\sum_{n'=n}^{N}\int_{\RR}\left((1+\ell_{n'}^{2})(\px Q_{n'})^{2}+Q_{n'}^{2}-2F(Q_{n'})\right)\left(\chi_{n}-1\right)\d x,
\\
\tilde{E}_{n}^{10}
&=\sum_{n'\ne n''}
\int_{\RR}\left((1+\ell_{n'}\ell_{n''})(\px Q_{n'})(\px Q_{n''})+Q_{n'}Q_{n''}\right)\chi_{n}\d x.
\end{align*}
By~\eqref{tech2} and~\eqref{est:theta},
\begin{equation*}
 |\tilde{E}_{n}^{7} |
\lesssim \int_{\RR}\big|F\left(U_{1}\right)-\sum_{n'=1}^{N}F\left(Q_{n'}\right)\big|\d x\lesssim \theta
\lesssim e^{-3\gamma_{0}L}.
\end{equation*}
From $1<q<p<\infty$,~\eqref{est:n'len2} and~\eqref{est:nlen'2},
\begin{align*}
 |\tilde{E}_{n}^{8} |&\lesssim \sum_{n'=1}^{n-1}\int_{\RR}\left((\px Q_{n'})^{2}+Q^{2}_{n'}+|Q_{n'}|^{p+1}+|Q_{n'}|^{q+1}\right)\chi_{n}\d x\\
&\lesssim \sum_{n'=1}^{n-1}\int_{\RR}\left((\px Q_{n'})^{2}+Q^{2}_{n'}\right)\chi_{n}\d x\lesssim e^{-4\gamma_{0}L},
\end{align*}
and
\begin{align*}
 |\tilde{E}_{n}^{9} |&\lesssim \sum_{n'=n}^{N}\int_{\RR}\left((\px Q_{n'})^{2}+Q^{2}_{n'}+|Q_{n'}|^{p+1}+|Q_{n'}|^{q+1}\right)(1-\chi_{n})\d x\\
&\lesssim \sum_{n'=n}^{N}\int_{\RR}\left((\px Q_{n'})^{2}+Q^{2}_{n'}\right)(1-\chi_{n})\d x\lesssim e^{-4\gamma_{0}L}.
\end{align*}
Last, using again~\eqref{tech1} and~\eqref{est:theta}, we have
\begin{equation*}
 |\tilde{E}_{n}^{10} |
\lesssim \sum_{n'\ne n''}\int_{\RR}\left(\left|(\px Q_{n'})(\px Q_{n''})\right|+\left|Q_{n'}Q_{n''}\right|\right)\d x\lesssim \theta
\lesssim e^{-3\gamma_{0}L}.
\end{equation*}
We see that~\eqref{est:En} follows from above estimates.

{\bf{Step 5.}} Estimate of $F_{n}(\vec{u})$. We claim
\begin{equation}\label{est:Fn}
\left|F_{n}(\vec{u})\right|\lesssim \frac{\|\vec{\varphi}\|^{2}_{\E}}{L}+\|\vec{\varphi}\|_{\E}^{3}+e^{-3\gamma_{0}L}.
\end{equation}
From~\eqref{def:psi}, we decompose
\begin{equation*}
F_{n}(\vec{u})=F_{n}^{1}+F_{n}^{2}+F_{n}^{3},
\end{equation*}
where
\begin{align*}
F_{n}^{1}&=-\frac{\alpha}{t+a}\int_{\RR}\left(x-\beta_{n}t-\bar{y}^{0}_{n}\right)(\vp\vpp) (\px \chi_{n} )\d x,\\
F_{n}^{2}&=\frac{\alpha}{t+a}\sum_{n',n''=1}^{N}\ell_{n''}
\int_{\RR}\left( x-\beta_{n}t-\bar{y}^{0}_{n}\right)\left(Q_{n'}\px Q_{n''}\right) (\px \chi_{n} )\d x,
\\
F_{n}^{3}&=-\frac{\alpha}{t+a}\sum_{n'=1}^{N}\int_{\RR}\left(x-\beta_{n}t-\bar{y}^{0}_{n}\right)\left(\vpp Q_{n'}-\ell_{n'}\vp \px Q_{n'}\right) (\px \chi_{n} )\d x.
\end{align*}
First, from~\eqref{pxchi} and $a=\left(\frac{L}{10}\right)^{\frac{1}{\alpha}}>\frac{L}{10}$,
\begin{equation*}
\left|F_{n}^{1}\right|\lesssim \frac{1}{(t+a)}\int_{\RR}|\vp||\vpp|\d x\lesssim \frac{\|\vec{\varphi}\|^{2}_{\E}}{L}.
\end{equation*}
Next, using again~\eqref{pxchi},~\eqref{est:1lenleN}, the AM-GM inequality and the Young's inequality,
\begin{align*}
\left|F_{n}^{2}\right|
&\lesssim \frac{1}{(t+a)}\sum_{n',n''=1}^{N}\int_{\RR}\left|Q_{n'}\right|\left|\px Q_{n''}\right|\textbf{1}_{\Omega_{n}}\d x\\
&\lesssim \sum_{n'=1}^{N}\left(\|\px Q_{n'}\textbf{1}_{\Omega_{n}}\|^{2}_{L^{2}}+\|Q_{n'}\textbf{1}_{\Omega_{n}}\|^{2}_{L^{2}}\right)\lesssim
e^{-4\gamma_0L},
\end{align*}
and
\begin{align*}
\left|F_{n}^{3}\right|&\lesssim \frac{1}{(t+a)}\sum_{n'=1}^{N}\int_{\RR}\left(\left|\vpp Q_{n'}\right|+\left|\vp \px Q_{n'}\right|\right)\textbf{1}_{\Omega_{n}}\d x\\
&\lesssim \sum_{n'=1}^{N}\|\vec{\varphi}\|_{\E}\left(\|\px Q_{n'}\textbf{1}_{\Omega_{n}}\|_{L^{2}}+\|Q_{n'}\textbf{1}_{\Omega_{n}}\|_{L^{2}}\right)\\
&\lesssim \sum_{n'=1}^{N}\left(\|\px Q_{n'}\textbf{1}_{\Omega_{n}}\|^{\frac{3}{2}}_{L^{2}}+\|Q_{n'}\textbf{1}_{\Omega_{n}}\|^{\frac{3}{2}}_{L^{2}}+\|\vec{\varphi}\|^{3}_{\E}\right)\lesssim \|\vec{\varphi}\|^{3}_{\E}+
e^{-3\gamma_0L}.
\end{align*}
Gathering above estimates, we obtain~\eqref{est:Fn}.

{\bf{Step 6.}} Conclude. First, we claim 
\begin{equation}\label{est:E+En}
E(\vec{u})+\sum_{n=2}^{N}c_{n}\beta_{n}E_{n}(\vec{u})=\mathcal{E}_{1}+\mathcal{E}_{2}+\mathcal{E}_{3}
+O\left(\|\vec{\varphi}\|_{\E}^{3}+\|\vec{\varphi}\|_{\E}^{q+1}+e^{-3\gamma_{0}L}\right).
\end{equation}
where
\begin{align*}
\mathcal{E}_{1}&=\sum_{n=1}^{N}\tilde{c}_{n}(1-\ell^{2}_{n})^{\frac{1}{2}}E(\vec{Q}),\\
\mathcal{E}_{2}&=2\sum_{n=1}^{N}\tilde{c}_{n}\ell_{n}\int_{\RR}(\px Q_{n})
\left(\ell_{n}\px Q_{n}+\ell_{n}\px \varphi_{1}-\varphi_{2}\right)\d x,\\
\mathcal{E}_{3}&=\sum_{n=1}^{N}\tilde{c}_{n}\int_{\RR}\big((\px \varphi_{1})^{2}+\varphi_{1}^{2}+\varphi_{2}^{2}-f'(Q_{n})\varphi_{1}^{2}\big)\psi_{n} \d x.
\end{align*}
Indeed, from~\eqref{def:tc},~\eqref{est:E} and~\eqref{est:En} and direct computation,
\begin{align*}
&E(\vec{u})+\sum_{n=2}^{N}c_{n}\beta_{n}E_{n}(\vec{u})\\
&=\mathcal{E}_{1}+\mathcal{E}_{2}+\int_{\RR}\bigg(1+\sum_{n=1}^{N}c_{n}\beta_{n}\chi_{n}\bigg)\bigg((\px \varphi_{1})^{2}+\varphi_{1}^{2}+\varphi_{2}^{2}\bigg)\d x\\
&\quad -\sum_{n=1}^{N}\bigg[\int_{\RR}\bigg(1+\sum_{n'=1}^{n}c_{n'}\beta_{n'}\chi_{n'}\bigg)f'(Q_{n})\varphi_{1}^{2}\d x\bigg] +O\left(\|\vec{\varphi}\|_{\E}^{3}+\|\vec{\varphi}\|_{\E}^{q+1}+e^{-3\gamma_{0}L}\right).
\end{align*} 
Observe that, from~\eqref{decompsi} and the definition of $\tilde{c}_{n}$,
\begin{align*}
&\int_{\RR}\bigg(1+\sum_{n=1}^{N}c_{n}\beta_{n}\chi_{n}\bigg)\bigg((\px \varphi_{1})^{2}+\varphi_{1}^{2}+\varphi_{2}^{2}\bigg)\d x\\
&\qquad =\sum_{n=1}^{N}\tilde{c}_{n}\int_{\RR}\bigg((\px \varphi_{1})^{2}+\varphi_{1}^{2}+\varphi_{2}^{2}\bigg)\psi_{n} \d x,
\end{align*}
and
\begin{align*}
&\sum_{n=1}^{N}\bigg[\int_{\RR}\big(1+\sum_{n'=2}^{n}c_{n'}\beta_{n'}\chi_{n'}\big)f'(Q_{n})\varphi_{1}^{2}\d x\bigg]\\
&\quad =\sum_{n=1}^{N}\tilde{c}_{n}\int_{\RR}f'(Q_{n})\varphi_{1}^{2}\psi_{n} \d x +\sum_{n=1}^{N}\sum_{n'\ne n}\bigg[\bigg(1+\sum_{k=1}^{(n')^{+}}(c_{k}\beta_{k})\bigg)
\int_{\RR}f'(Q_{n})\varphi_{1}^{2}\psi_{n'} \d x\bigg],
\end{align*}
where $(n')^{+}=\min(n',n)$. Note that, by~\eqref{est:f'nnen},~the Cauchy-Schwarz inequality, the Young's inequality and the Sobolev embedding Theorem,
\begin{align*}
\sum_{n'\ne n}\bigg|
\int_{\RR}f'(Q_{n})\varphi_{1}^{2}\psi_{n'} \d x\bigg|&\lesssim \sum_{n'\ne n}\|\vec{\varphi}\|_{\E}^{2}\|f'(Q_{n})\psi_{n'}\|_{L^{2}}\\
&\lesssim \|\vec{\varphi}\|_{\E}^{3}+\sum_{n'\ne n}\|f'(Q_{n})\psi_{n'}\|_{L^{2}}^{3}\lesssim \|\vec{\varphi}\|_{\E}^{3}+e^{-6\gamma_{0}L}.
\end{align*}
We see that~\eqref{est:E+En} follows from combining these identities and estimates.

Second, we claim 
\begin{equation}\label{est:I+In}
\begin{aligned}
 c_{1}\mathcal{I}+\sum_{n=2}^{N}c_{n}\mathcal{I}_{n} &=
-2\sum_{n=1}^{N}\bigg[\bigg(\sum_{n'=1}^{n}c_{n'}\bigg)\int_{\RR}(\px Q_{n})\left(\ell_{n}\px Q_{n}+\ell_{n}\px \varphi_{1}-\varphi_{2}\right)\d x\bigg]\\
&\qquad +2\sum_{n=1}^{N}\bigg[\bigg(\sum_{n'=1}^{n}c_{n'}\bigg)\int_{\RR}(\px \vp)\vpp \psi_{n}\d x\bigg]\\
&\qquad +
O\left(\frac{\|\vec{\varphi}\|^{2}_{\E}}{L}+\|\vec{\varphi}\|_{\E}^{3}+e^{-3\gamma_{0}L}\right).
\end{aligned}
\end{equation}
Indeed, from~\eqref{est:I} and~\eqref{est:In},
\begin{align*}
&c_{1}\mathcal{I}+\sum_{n=2}^{N}c_{n}\mathcal{I}_{n}\\
&\quad =-2\sum_{n=1}^{N}\int_{\RR}\bigg(\sum_{n'=1}^{n}c_{n'}\bigg)(\px Q_{n})(\ell_{n}\px Q_{n}+\ell_{n}\px \varphi_{1}-\varphi_{2})\d x\\
&\qquad +2\int_{\RR}\bigg(c_{1}+\sum_{n=2}^{N}c_{n}\chi_{n}\bigg)(\px \varphi_{1})\vpp \d x+O\left(\frac{\|\vec{\varphi}\|^{2}_{\E}}{L}+\|\vec{\varphi}\|_{\E}^{3}+e^{-3\gamma_{0}L}\right).
\end{align*}
Using again~\eqref{decompsi}, we have
\begin{equation*}
2\int_{\RR}\bigg(c_{1}+\sum_{n=2}^{N}c_{n}\chi_{n}\bigg)(\px \varphi_{1})\vpp \d x
=2\sum_{n=1}^{N}\bigg[\bigg(\sum_{n'=1}^{n}c_{n'}\bigg)\int_{\RR}(\px \vp)\vpp \psi_{n}\d x\bigg],
\end{equation*}
which implies~\eqref{est:I+In}.

Combining~\eqref{eq:sumck2},~\eqref{def:Jn},~\eqref{def:comE},~\eqref{est:Fn},~\eqref{est:E+En} and~\eqref{est:I+In}, we obtain~\eqref{expan:E}.
\end{proof}

Next, using the standard localized argument, we obtain the following coercivity result.
\begin{lemma}[Coercivity]
There exist $0<\mu\ll 1$ such that for all $t\in[0,T_{*}(\vec{u}_{0})]$,
\begin{equation}\label{coer:HN}
\mu \|\vec{\varphi}\|_{\E}^{2}
\le \sum_{n=1}^{N}\tilde{c}_{n} H_{n}(\vec{\varphi},\vec{\varphi})+\mu^{-1}
\Bigg[\sum_{n=1}^{N}(a_{n}^{-})^{2}+\sum_{n=1}^{N}(a_{n}^{+})^{2}\Bigg]+e^{-3\gamma_0 L}.
\end{equation}
\end{lemma}
\begin{proof}
First, we prove the localized coercivity of $H_{n}(\vec{\varphi},\vec{\varphi})$ from the coercivity property~\eqref{coer:H} around one solitary wave with the orthogonality properties~\eqref{phiorth}. By direct computation, the defintion of $\psi_{n}$ and $\sqrt{\chi},\sqrt{1-\chi}\in C^{1}$,
\begin{equation}\label{est:coerloc}
\begin{aligned}
&\int_{\R}\left((\px \vp)^{2}+2\ell_{n}(\px \vp)\vpp\right)\psi_{n}\d x\\
=&\int_{\R}\left(\left(\px \left(\vp \sqrt{\psi_{n}}\right)\right)^{2}+2\ell_{n}\left(\px \left(\vp\sqrt{\psi_{n}}\right)\right)\vpp\sqrt{\psi_{n}}\right)\d x+O\left(\frac{1}{L}\|\vec{\varphi}\|_{\E}^{2}\right).
\end{aligned}
\end{equation}
We define
\begin{equation*}
\vec{\varphi}_{n}=\left(\varphi_{1,n},\varphi_{2,n}\right)=\left(\varphi_{1}(\cdot+y_{n})\sqrt{\psi_{n}(\cdot+y_{n})},\varphi_{2}(\cdot+y_{n})\sqrt{\psi_{n}(\cdot+y_{n})}\right),
\end{equation*}
for all $n=1,\ldots,N$. From~\eqref{coer:H},~\eqref{est:coerloc} and the translation $x\to x+y_{n}$, there exists $\nu_{n}>0$ such that 
\begin{align*}
\mathcal{H}_{n}\left(\vec{\varphi},\vec{\varphi}\right)
&=\left(\mathcal{H}_{\ell_{n}}\vec{\varphi}_{n},\vec{\varphi}_{n}\right)_{L^{2}}+O\left(\frac{1}{L}\|\vec{\varphi}\|_{\E}^{2}\right)\\
&\ge \nu_{n}\|\vec{\varphi}_{n}\|^{2}_{\mathcal{H}}-\nu^{-1}_{n}\left(\left(\vec{\varphi}_{n},\vec{Z}^{0}_{\ell_{n}}\right)^{2}_{L^{2}}+\sum_{\pm}\left(\vec{\varphi}_{n},\vec{Z}^{\pm}_{\ell_{n}}\right)^{2}_{L^{2}}\right)+O\left(\frac{1}{L}\|\vec{\varphi}\|_{\E}^{2}\right).
\end{align*}
By the definition of $\psi_{n}$ and the translation $x\to x-y_{n}$,
\begin{align*}
\|\vec{\varphi}_{n}\|_{\mathcal{H}}^{2}=\int_{\R}\left((\px \vp)^{2}+\vp^{2}+\vpp^{2}\right)\psi_{n}\d x+O\left(\frac{1}{L}\|\vec{\varphi}\|_{\E}^{2}\right).
\end{align*}
Note that, from $0\le \psi_{n}\le 1$, we have ${1-\sqrt{\psi_{n}}}\le 1-\psi_{n}$. Thus, using~\eqref{est:Upsilonpsi},~\eqref{est:Qpsi}, and $\sum_{n=1}^{N}\psi_{n}=1$, we have 
\begin{equation*}
\begin{aligned}
&\|(1-\sqrt{\psi_{n}})\px Q_{n}\|_{L^{2}}+\|(1-\sqrt{\psi_{n}})\partial_{x}^{2} Q_{n}\|_{L^{2}}\\
\lesssim& \|(1-{\psi_{n}})\px Q_{n}\|_{L^{2}}+\|(1-{\psi_{n}})\partial_{x}^{2} Q_{n}\|_{L^{2}}\\
\lesssim& \sum_{n' \ne n}\|\px Q_{n}\psi_{n'}\|_{L^{2}}+\sum_{n' \ne n}\|\partial_{x}^{2} Q_{n}\psi_{n'}\|_{L^{2}}
\lesssim e^{-2\gamma_{0} L},
\end{aligned}
\end{equation*}
and
\begin{equation*}
\begin{aligned}
&\|(1-\sqrt{\psi_{n}})\Upsilon_{n,1}\|_{L^{2}}+\|(1-\sqrt{\psi_{n}})\Upsilon_{n,2}\|_{L^{2}}\\
\lesssim &\sum_{n' \ne n}\|\Upsilon_{n,1}\psi_{n'}\|_{L^{2}}+\sum_{n' \ne n}\|\Upsilon_{n,2}\psi_{n'}\|_{L^{2}}
\lesssim e^{-2\gamma_{0} L}.
\end{aligned}
\end{equation*}
Based on above estimates,~\eqref{phiorth}, the definition of $a_{n}$, the translation $x\to x-y_{n}$ and the Young's inequality, we have 
\begin{align*}
\left(\vec{\varphi}_{n},\vec{Z}^{0}_{\ell_{n}}\right)^{2}_{L^{2}}
&=\left(\vec{\varphi},(1-\sqrt{\psi_{n}})\vec{Z}^{0}_{n}\right)_{L^{2}}^{2}\\
&=O\left(\|\vec{\varphi}\|_{\E}^{3}+\|(1-\sqrt{\psi_{n}})\px Q_{n}\|^{6}_{L^{2}}+\|(1-\sqrt{\psi_{n}})\partial_{x}^{2} Q_{n}\|^{6}_{L^{2}}\right)\\
&=O\left(\|\vec{\varphi}\|_{\E}^{3}+e^{-12\gamma_0L}\right),
\end{align*}
and
\begin{align*}
\left(\vec{\varphi}_{n},\vec{Z}^{\pm}_{\ell_{n}}\right)^{2}_{L^{2}}&=\left(a_{n}^{\pm}+\left(\vec{\varphi},(1-\sqrt{\psi_{n}})\vec{Z}^{\pm}_{n}\right)_{L^{2}}\right)^{2}\\
&=(a_{n}^{\pm})^{2}+O\left(\|\vec{\varphi}\|_{\E}^{3}+\|(1-\sqrt{\psi_{n}})\Upsilon_{n,1}\|^{3}_{L^{2}}+\|(1-\sqrt{\psi_{n}})\Upsilon_{n,2}\|^{3}_{L^{2}}\right)\\
&=(a_{n}^{\pm})^{2}+O\left(\|\vec{\varphi}\|_{\E}^{3}+e^{-6\gamma_{0} L}\right).
\end{align*}
Gathering above estimates, we obtain the following localized coercivity of $H_{n}(\vec{\varphi},\vec{\varphi})$,
\begin{equation}\label{coerloc}
\begin{aligned}
H_{n}(\vec{\varphi},\vec{\varphi})
&\ge \nu_{n}\int_{\R}\left((\px \vp)^{2}+\vp^{2}+\vpp^{2}\right)\psi_{n}\d x-\nu_{n}^{-1}\left((a_{n}^{-})^{2}+(a_{n}^{+})^{2}\right)\\
&\quad +O\left(\frac{1}{L}\|\vec{\varphi}\|_{\E}^{2}+\|\vec{\varphi}\|_{\E}^{3}+e^{-6\gamma_0 L}\right), 
\end{aligned}
\end{equation}
for all $n=1,\ldots,N$.
Last, combining $\tilde{c}_{n}>0$,~\eqref{coerloc}, $\sum_{n=1}^{N}\psi_{n}=1$ and taking $L$ large enough, we obtain the coercivity property~\eqref{coer:HN}.
\end{proof}

Last, we prove the following almost monotonicity property for $\mathcal{J}_{n}$.
\begin{lemma}\label{le:virial}
There exist $C_{1}$ such that for any $n=1,\ldots,N$ and $t\in [0,T_{*}(\vec{u}_{0})]$,
\begin{equation}\label{mon:I}
\mathcal{J}_{n}(t)-\mathcal{J}_{n}(0)\le \frac{C_{1}}{L^{2\alpha-1}}\sup_{s\in[0,t]}\|\vec{\varphi}(s)\|_{\E}^{2}+C_{1}e^{-3\gamma_{0}L}.
\end{equation}
\end{lemma}
\begin{proof}
{\bf Step 1}. Time variation of $\mathcal{I}_{n}$. We claim
\begin{equation}\label{estI}
\begin{aligned}
\frac{\d}{\d t}\mathcal{I}_{n}
&=-2\int_{\RR}(\px u_{1}+\beta_{n}u_{2})^{2}\px \chi_{n}\d x\\
&\quad +(1-\beta_{n}^{2})\int_{\RR}\big(u_{1}f(u_{1})-2F(u_{1})\big)\px \chi_{n}\d x\\
&\quad+\beta_{n}^{2}\int_{\RR}\big((\px u_{1})^{2}+u_{1}^{2}+u_{2}^{2}\big)\px \chi_{n}\d x
\\
&\quad +2\beta_{n}\int_{\RR}(\px u_{1})u_{2}\px \chi_{n}\d x
-2\beta_{n}^{2}\int_{\RR}F(u_{1})\px \chi_{n}\d x\\
&\quad -\frac{\alpha}{(t+a)^{1-\alpha}}\int_{\RR}\left(\frac{x-\beta_{n}t-\bar{y}^{0}_{n}}{(t+a)^{\alpha}}\right)\big(2(\px u_{1})u_{2}\big)\px \chi_{n}\d x\\
&\quad+O\left(\frac{1}{(t+a)^{2\alpha}}\|\vec{\varphi}\|_{\E}^{2}+e^{-4\left(\gamma_{0} L+\gamma_{1} t\right)}\right).
\end{aligned}
\end{equation}
First, using~\eqref{NLKGvec} and integrating by parts,
\begin{align*}
\frac{\rm{d}}{{\rm d}t}\int_{\RR}2\big(\chi_{n}\px u_{1}\big)u_{2}\d x
&=-\int_{\RR}\big((\px u_1)^{2}-u_{1}^{2}+u_{2}^{2}+2F(u_1)\big)\px \chi_{n}\d x\\
&\quad +2\int_{\RR}(\pt \chi_{n})(\px u_1)u_{2}\d x.
\end{align*}
Observe that 
\begin{equation}\label{ptchin}
\pt \chi_{n}=-\beta_{n}\px \chi_{n}-\frac{\alpha}{(t+a)^{1-\alpha}}\left(\frac{x-\beta_{n}t-\bar{y}^{0}_{n}}{(t+a)^{\alpha}}\right)\left(\px \chi_{n}\right).
\end{equation}
Therefore,
\begin{equation}\label{estI1}
\begin{aligned}
\frac{\rm{d}}{{\rm d}t}\int_{\RR}2\big(\chi_{n}\px u_{1}\big)u_{2}\d x=&
-\int_{\RR}\left((\px u_{1})^{2}-u_{1}^{2}+u_{2}^{2}+2\beta_{n}(\px u_{1})u_{2}\right)\px \chi_{n}\d x\\
&-\frac{\alpha}{(t+a)^{1-\alpha}}\int_{\RR}\left(\frac{x-\beta_{n}t-\bar{y}^{0}_{n}}{(t+a)^{\alpha}}\right)\big(2(\px u_{1})u_{2}\big)\px \chi_{n}\d x\\
&-2\int_{\RR}F(u_{1})\px \chi_{n} \d x.
\end{aligned}
\end{equation}
Second, using again~\eqref{NLKGvec} and integrating by parts,
\begin{align*}
\frac{\rm{d}}{{\rm d}t}\int_{\RR}u_{1}u_{2}\px \chi_{n}\d x
&=\int_{\RR}
\left(-(\px u_1)^{2}-u_{1}^{2}+u_{2}^{2}+u_{1}f(u_{1})\right)\px \chi_{n}\d x\\
&+\int_{\RR}(u_{1}u_{2})\partial_{tx} \chi_{n}\d x+\frac{1}{2}\int_{\RR}u_{1}^{2}\px ^{3}\chi_{n} \d x.
\end{align*}
 From~\eqref{def:psi},~\eqref{est:px2chi},~\eqref{est:1lenleN} and the AM-GM inequality, we have 
\begin{align*}
\left|\int_{\RR}(u_{1}u_{2})\partial_{tx} \chi_{n}\d x\right|
&\lesssim \frac{1}{(t+a)^{2\alpha}}\int_{\RR} \bigg(\sum_{n'=1}^{N}(Q_{n'})^{2}+\sum_{n'=1}^{N}(\px Q_{n'})^{2}+\varphi_{1}^{2}+\vpp^{2}\bigg) \textbf{1}_{\Omega_{n}}\d x \\
&\lesssim \frac{1}{(t+a)^{2\alpha}}\|\vec{\varphi}\|_{\E}^{2}+\sum_{n'=1}^{N}\left(\|\px Q_{n'}\textbf{1}_{\Omega_{n}}\|^{2}_{L^{2}}+\|Q_{n'}\textbf{1}_{\Omega_{n}}\|^{2}_{L^{2}}\right)\\
&\lesssim \frac{1}{(t+a)^{2\alpha}}\|\vec{\varphi}\|_{\E}^{2}+e^{-4(\gamma_{0}L+\gamma_{1}t)},
\end{align*}
and
\begin{align*}
\left|\int_{\RR}u_{1}^{2}\px ^{3}\chi_{n} \d x\right|&\lesssim \frac{1}{(t+a)^{3\alpha}}\int_{\RR}\left(\sum_{n'=1}^{N}(Q_{n'})^{2}+\sum_{n'=1}^{N}(\px Q_{n'})^{2}+\varphi_{1}^{2}\right)\textbf{1}_{\Omega_{n}}\d x\\
&\lesssim \frac{1}{(t+a)^{3\alpha}}\|\vec{\varphi}\|_{\E}^{2}+\sum_{n'=1}^{N}\left(\|\px Q_{n'}\textbf{1}_{\Omega_{n}}\|^{2}_{L^{2}}+\|Q_{n'}\textbf{1}_{\Omega_{n}}\|^{2}_{L^{2}}\right)
\\
&\lesssim \frac{1}{(t+a)^{3\alpha}}\|\vec{\varphi}\|_{\E}^{2}+e^{-4(\gamma_{0}L+\gamma_{1}t)}.
\end{align*}
It follows that
\begin{equation}\label{estI2}
\begin{aligned}
\frac{\rm{d}}{{\rm d}t}\int_{\RR}u_{1}u_{2}\px \chi_{n}\d x
&=\int_{\RR}
\left(-(\px u_1)^{2}-u_{1}^{2}+u_{2}^{2}+u_{1}f(u_{1})\right)\px \chi_{n}\d x\\
&+O\left(\frac{1}{(t+a)^{2\alpha}}\|\vec{\varphi}\|_{\E}^{2}+e^{-4(\gamma_{0}L+\gamma_{1}t)}\right).
\end{aligned}
\end{equation}
Gathering~\eqref{estI1} and~\eqref{estI2}, we obtain~\eqref{estI}.

{\bf{Step 2.}} Time variation of $E_{n}$. We claim
\begin{equation}\label{estE}
\begin{aligned}
\frac{\d}{\d t}E_{n}
&=-2\int_{\RR}\big((\px u_{1})u_{2}\big)\px \chi_{n}\d x-\beta_{n}\int_{\RR}\big((\px u_{1})^{2}+u_{1}^{2}+u_{2}^{2}\big)\px \chi_{n}\d x\\
&\quad+2\beta_{n}\int_{\RR}F(u_{1})\px \chi_{n}\d x\\
&\quad +\frac{2\alpha}{(t+a)^{1-\alpha}}\int_{\RR}\left(\frac{x-\beta_{n}t-\bar{y}^{0}_{n}}{(t+a)^{\alpha}}\right)F(u_{1})\px \chi_{n}\d x\\
&\quad-\frac{\alpha}{(t+a)^{1-\alpha}}\int_{\RR}\left(\frac{x-\beta_{n}t-\bar{y}^{0}_{n}}{(t+a)^{\alpha}}\right)
\big((\px u_{1})^{2}+u_{1}^{2}+u_{2}^{2}\big)\px \chi_{n}\d x.
\end{aligned}
\end{equation}

First, from~\eqref{NLKGvec},
\begin{equation*}
\pt \big[(\px u_{1})^{2}+u_{1}^{2}+u_{2}^{2}-2F(u_{1})\big]
=2(\px u_{1})\px u_{2}+2(\px^{2} u_{1})u_{2}=2\px (u_{2}\px u_{1}).
\end{equation*}
Therefore, by integration by parts,
\begin{equation*}
 \int_{\RR}\big(\pt \big[(\px u_{1})^{2}+u_{1}^{2}+u_{2}^{2}-2F(u_{1})\big]\big)\chi_{n}\d x
 =-2\int_{\RR}\big((\px u_{1})u_{2}\big)\px \chi_{n}\d x.
\end{equation*}
Second, from~\eqref{ptchin},
\begin{align*}
&\int_{\RR}\left((\px u_{1})^{2}+u_{1}^{2}+u_{2}^{2}-2F(u_{1})\right)\pt \chi_{n}\d x\\
&\quad =-\beta_{n}\int_{\RR}\left((\px u_{1})^{2}+u_{1}^{2}+u_{2}^{2}-2F(u_{1})\right)\px \chi_{n}\d x\\
 &\qquad -\frac{\alpha}{(t+a)^{1-\alpha}}\int_{\RR}\left(\frac{x-\beta_{n}t-\bar{y}^{0}_{n}}{(t+a)^{\alpha}}\right)\left((\px u_{1})^{2}+u_{1}^{2}+u_{2}^{2}-2F(u_{1})\right)
 \px \chi_{n}\d x.
\end{align*}
Gathering these identities, we obtain~\eqref{estE}.

{\bf{Step 3.}} Time variation of $F_{n}$. We claim
\begin{equation}\label{estF}
\begin{aligned}
\frac{\d}{\d t}F_{n}
&=\frac{-\alpha}{(t+a)^{1-\alpha}}\int_{\RR}\left(\frac{x-\beta_{n}t-\bar{y}^{0}_{n}}{(t+a)^{\alpha}}\right)\bigl[-(\px u_{1})^{2}-u_{1}^{2}+u_{2}^{2}+u_{1}f(u_{1})\bigr]
\px \chi_{n}\d x\\
&\quad +O\left(\frac{1}{(t+a)^{1+\alpha}}\|\vec{\varphi}\|_{\E}^{2}+e^{-4(\gamma_{0}L+\gamma_{1}t)}\right).
\end{aligned}
\end{equation}
First, from~\eqref{NLKGvec},
$
\pt (u_{1}u_{2})=u_{2}^{2}-u_{1}^{2}+u_{1}f(u_{1})+u_{1}(\px ^{2}u_{1}).
$
Therefore, by integration by parts,
\begin{align*}
&\frac{1}{(t+a)^{1-\alpha}}\int_{\RR}\left(\frac{x-\beta_{n}t-\bar{y}^{0}_{n}}{(t+a)^{\alpha}}\right)\big[\pt (u_{1}u_{2})\big]\px \chi_{n}\d x\\
&\quad =\frac{1}{(t+a)^{1-\alpha}}\int_{\RR}\left(\frac{x-\beta_{n}t-\bar{y}^{0}_{n}}{(t+a)^{\alpha}}\right)\left(-(\px u_{1})^{2}-u_{1}^{2}+u_{2}^{2}+u_{1}f(u_{1})\right)\px \chi_{n}\d x\\
&\qquad+\frac{1}{t+a}\int_{\RR}u_{1}^{2}\px^{2}\chi_{n}\d x
+\frac{1}{2(t+a)^{1-\alpha}}\int_{\RR}\left(\frac{x-\beta_{n}t-\bar{y}^{0}_{n}}{(t+a)^{\alpha}}\right)u_{1}^{2}\px^{3}\chi_{n}\d x.
\end{align*}
Observe that, from~\eqref{def:psi},~\eqref{est:px2chi},~\eqref{est:1lenleN} and the AM-GM inequality, we have 
\begin{align*}
\left|\frac{1}{t+a}\int_{\RR}u_{1}^{2}\px^{2}\chi_{n}\d x\right|
&\lesssim \frac{1}{(t+a)^{1+2\alpha}}\int_{\RR}
\left(\sum_{n'=1}^{N}Q_{n'}^{2}+\vp^{2}\right)\textbf{1}_{\Omega_{n}}\d x\\
&\lesssim \frac{1}{(t+a)^{1+2\alpha}}\|\vec{\varphi}\|_{\E}^{2}
+\sum_{n'=1}^{N}\|Q_{n'}\textbf{1}_{\Omega_{n}}\|_{L^{2}}^{2}\\
&\lesssim \frac{1}{(t+a)^{1+2\alpha}}\|\vec{\varphi}\|_{\E}^{2}
+e^{-4(\gamma_{0}L+\gamma_{1}t)},
\end{align*}
and
\begin{align*}
&\left|\frac{1}{2(t+a)^{1-\alpha}}\int_{\RR}\left(\frac{x-\beta_{n}t-\bar{y}^{0}_{n}}{(t+a)^{\alpha}}\right)u_{1}^{2}\px^{3}\chi_{n}\d x\right|\\
&\quad \lesssim \frac{1}{(t+a)^{1+2\alpha}}\int_{\RR}
\left(\sum_{n'=1}^{N}Q_{n'}^{2}+\vp^{2}\right)\textbf{1}_{\Omega_{n}}\d x\\
&\quad \lesssim \frac{1}{(t+a)^{1+2\alpha}}\|\vec{\varphi}\|_{\E}^{2}
+e^{-4(\gamma_{0}L+\gamma_{1}t)}.
\end{align*}
It follows that 
\begin{equation}\label{estFn1}
\begin{aligned}
&\frac{1}{(t+a)^{1-\alpha}}\int_{\RR}\left(\frac{x-\beta_{n}t-\bar{y}^{0}_{n}}{(t+a)^{\alpha}}\right)\big[\pt (u_{1}u_{2})\big]\px \chi_{n}\d x\\
&\quad=\frac{1}{(t+a)^{1-\alpha}}\int_{\RR}\left(\frac{x-\beta_{n}t-\bar{y}^{0}_{n}}{(t+a)^{\alpha}}\right)\left[-(\px u_{1})^{2}-u_{1}^{2}+u_{2}^{2}+u_{1}f(u_{1})\right]\px \chi_{n}\d x\\
&\qquad+O\left(\frac{1}{(t+a)^{1+2\alpha}}\|\vec{\varphi}\|_{\E}^{2}+e^{-4(\gamma_{0}L+\gamma_{1}t)}\right).
\end{aligned}
\end{equation}
Second, by direct computation and taking partial derivative of~\eqref{ptchin} with respect to $x$,
\begin{align*}
\pt \left[\left(\frac{x-\beta_{n}t-\bar{y}^{0}_{n}}{t+a}\right)\px \chi_{n}\right]
=&-\frac{\beta_{n}}{(t+a)}\px \chi_{n}-\frac{1}{(t+a)^{2-\alpha}}\left(\frac{x-\beta_{n}t-\bar{y}^{0}_{n}}{(t+a)^{\alpha}}\right)\px \chi_{n}\\
&-\frac{\beta_{n}}{(t+a)^{1-\alpha}}\left(\frac{x-\beta_{n}t-\bar{y}^{0}_{n}}{(t+a)^{\alpha}}\right)
\px^{2}\chi_{n}\\
&-\frac{\alpha}{(t+a)^{2-2\alpha}}\left(\frac{x-\beta_{n}t-\bar{y}^{0}_{n}}{(t+a)^{\alpha}}\right)^{2}\px^{2} \chi_{n}.
\end{align*}
Therefore, from~\eqref{pxchi},~\eqref{est:px2chi} and $\frac{1}{2}<\alpha<\frac{4}{7}$,
\begin{equation*}
\left|\pt \left[\left(\frac{x-\beta_{n}t-y^{0}_{n}}{t+a}\right)\px \chi_{n}\right]\right|\lesssim
\frac{1}{(t+a)^{1+\alpha}}\textbf{1}_{\Omega_{n}}.
\end{equation*}
 Based on~\eqref{def:psi},~\eqref{est:1lenleN}, the AM-GM inequality and above inequality, we have 
\begin{equation}\label{estFn2}
\begin{aligned}
&\left|\int_{\RR}(u_{1}u_{2})\pt \left[\left(\frac{x-\beta_{n}t-\bar{y}^{0}_{n}}{t+a}\right)\px \chi_{n}\right]\d x\right|\\
&\quad \lesssim \frac{1}{(t+a)^{1+\alpha}}\int_{\RR}\left(\sum_{n'=1}^{N}Q_{n'}^{2}+\sum_{n'=1}^{N}(\px Q_{n'})^{2}+\vp^{2}+\vpp^{2}\right)\textbf{1}_{\Omega_{n}}\d x\\
&\quad\lesssim \frac{1}{(t+a)^{1+\alpha}}\|\vec{\varphi}\|_{\E}^{2}+e^{-4(\gamma_{0}L+\gamma_{1}t)}.
\end{aligned}
\end{equation}
We see that~\eqref{estF} follows from~\eqref{estFn1} and~\eqref{estFn2}.

{\bf {Step 4.}} Conclude. Note that from~\eqref{estI},~\eqref{estE},~\eqref{estF} and $\frac{1}{2}<\alpha<\frac{4}{7}$,
\begin{align*}
\frac{\d}{\d t}\mathcal{J}_{n}
&=\frac{\d}{\d t}\mathcal{I}_{n}+\beta_{n}\frac{\d}{\d t}E_{n}+\beta_{n}\frac{\d}{\d t}F_{n}\\
&=\mathcal{F}_{1}+\mathcal{F}_{2}+O\left(\frac{1}{(t+a)^{2\alpha}}\|\vec{\varphi}\|_{\E}^{2}
+e^{-4(\gamma_{0}L+\gamma_{1}t)}\right),
\end{align*}
where
\begin{align*}
\mathcal{F}_{1}
=&-2\int_{\RR}(\px u_{1}+\beta_{n}u_{2})^{2}\px \chi_{n}\d x\\
&-\frac{2\alpha}{(t+a)^{1-\alpha}}\int_{\RR}\left(\frac{x-\beta_{n}t-\bar{y}^{0}_{n}}{(t+a)^{\alpha}}\right)\left(u_{2}(\px u_{1}+\beta_{n}u_{2})\right)\px \chi_{n}\d x,
\end{align*}
and
\begin{align*}
\mathcal{F}_{2}=\int_{\RR}\left[1-\beta_{n}^{2}-\frac{\alpha\beta_{n}}{(t+a)^{1-\alpha}}
\left(\frac{x-\beta_{n}t-\bar{y}^{0}_{n}}{(t+a)^{\alpha}}\right)\right]\left(u_{1}f(u_{1})-2F(u_{1})\right)\px \chi_{n}\d x.
\end{align*}
 \emph{Estimates of $\mathcal{F}_{1}$.}
We claim
\begin{equation}\label{estF1}
\mathcal{F}_{1}\le \frac{1}{(t+a)^{2\alpha}}\|\vec{\varphi}\|_{\E}^{2}+e^{-3(\gamma_{0}L+\gamma_{1}t)}.
\end{equation}
Indeed, by the AM-GM inequality and $1-\alpha=(1-\frac{5}{4}\alpha)+\frac{1}{4}\alpha$,
\begin{align*}
&\left|\frac{2\alpha}{(t+a)^{1-\alpha}}\int_{\RR}\left(\frac{x-\beta_{n}t-\bar{y}^{0}_{n}}{(t+a)^{\alpha}}\right)\left(u_{2}(\px u_{1}+\beta_{n}u_{2})\right)\px \chi_{n}\d x\right|\\
&\quad \lesssim \frac{1}{(t+a)^{2-\frac{5}{2}\alpha}}\int_{\RR}u_{2}^{2}\px \chi_{n}\d x
+\frac{1}{(t+a)^{\frac{\alpha}{2}}}\int_{\RR}(\px u_{1}+\beta_{n}u_{2})^{2}\px \chi_{n}\d x.
\end{align*}
Moreover, using~\eqref{def:psi},~\eqref{pxchi},~\eqref{est:1lenleN} and the AM-GM inequality,
\begin{align*}
\left|\frac{1}{(t+a)^{2-\frac{5}{2}\alpha}}\int_{\RR}u_{2}^{2}\px \chi_{n}\d x\right|
&\lesssim
\frac{1}{(t+a)^{2-\frac{3}{2}\alpha}}\int_{\RR}
\left(\sum_{n'=1}^{N}(\px Q_{n'})^{2}+\vpp^{2}\right)\textbf{1}_{\Omega_{n}}\d x\\
&\lesssim \frac{1}{(t+a)^{2-\frac{3}{2}\alpha}}\|\vec{\varphi}\|_{\E}^{2}+e^{-4(\gamma_{0}L+\gamma_{1}t)}.
\end{align*}
 Note that, from $\frac{1}{2}<\alpha<\frac{4}{7}$, we have $2-\frac{3}{2}\alpha>2\alpha$. Thus, from $a=\left(\frac{L}{10}\right)^{\frac{1}{\alpha}}$ and taking $L$ large enough,
\begin{align*}
\mathcal{F}_{1}\le& -2\int_{\RR}(\px u_{1}+\beta_{n}u_{2})^{2}\px \chi_{n}\d x
+\frac{1}{L^{\frac{1}{4}}}\int_{\RR}(\px u_{1}+\beta_{n}u_{2})^{2}\px \chi_{n}\d x\\
&+\frac{1}{(t+a)^{2\alpha}}\|\vec{\varphi}\|_{\E}^{2}+e^{-3(\gamma_{0}L+\gamma_{1}t)},
\end{align*}
which implies~\eqref{estF1}.

\emph{Estimates of $\mathcal{F}_{2}$.}
We claim
\begin{equation}\label{estF2}
\mathcal{F}_{2}\le 0.
\end{equation}
First, observe that, for $L$ large enough, for any $x\in \Omega_{n}$,
\begin{equation*}
1-\beta_{n}^{2}-
\frac{\alpha\beta_{n}}{(t+a)^{1-\alpha}}\left(\frac{x-\beta_{n}t-y^{0}_{n}}{(t+a)^{\alpha}}\right)
\ge 1-\beta_{n}^{2}-\left|\alpha\beta_{n}\right|\left|\frac{10}{L}\right|^{\frac{1}{\alpha}-1}>0.
\end{equation*}
Second, recall that, for any $x\in \Omega_n$ and $n'=1,\cdots,N$,
\begin{equation*}
\begin{aligned}
|x-y_{n'}(t)|
&\ge |\ell_{n'}-\beta_{n}|t+|y_{n'}(0)-\bar{y}_{n}^{0}|-\gamma_{1}t-t^{\alpha}-\frac{L}{10}\\
&\ge 8\gamma_{1} t+\frac{9L}{20}-\left(2\gamma_{1}t+\frac{L}{5}\right)\ge \frac{L}{4},
\end{aligned}
\end{equation*}
for $L$ large enough.
Thus, from the decay properties of $Q$ in~\eqref{Qdec}, for any $t\in [0,T_{*}(\vec{u}_{0})]$,
\begin{equation*}
\sum_{n'=1}^{N}|Q_{n'}(t,x)|\textbf{1}_{\Omega_{n}}\lesssim \sum_{n'=1}^{N}e^{\frac{-|x-y_{n'}(t)|}{\sqrt{1-\ell_{n'}^{2}}}}\textbf{1}_{\Omega_{n}}\lesssim e^{-\frac{L}{4}}.
\end{equation*}
Based on above estimate,~\eqref{Bootset} and $\gamma_{0}\le \frac{1}{8}$, we have, for any $x\in \Omega_n$,
\begin{align*}
|u_{1}(t,x)|
&\lesssim \sum_{n'=1}^{N}|Q_{n'}(t,x)|\textbf{1}_{\Omega_{n}}+|\vp(t,x)|\\
&\lesssim e^{-\frac{L}{4}}+\|\vec{\varphi}\|_{\E}\lesssim C_{0}\left(e^{-\gamma_{0}L}+\delta\right).
\end{align*}
Therefore, for $L$ large enough and $\delta$ small enough, we obtain for any $x\in \Omega_n$,
\begin{equation*}
u_{1}f(u_{1})-2F(u_{1})=-\frac{q-1}{q+1}|u_{1}|^{q+1}+\frac{p-1}{p+1}|u_{1}|^{p+1}\le 0,
\end{equation*}
since $1<q<p<\infty$. We conclude~\eqref{estF2} from above estimates.

Gathering estimates~\eqref{estF1} and~\eqref{estF2}, we obtain
\begin{equation*}
\frac{\d}{\d t}\mathcal{J}_{n}(t)\lesssim  \frac{1}{(t+a)^{2\alpha}}\|\vec{\varphi}(t)\|_{\E}^{2}+e^{-3(\gamma_{0}L+\gamma_{1}t)}.
\end{equation*}
Integrating on $[0,t]$ for any $t\in [0,T_{*}(\vec{u}_{0})]$, we obtain 
\begin{equation*}
\mathcal{J}_{n}(t)- \mathcal{J}_{n}(0)\lesssim \frac{1}{L^{2\alpha-1}}\max_{s\in[0,t]}\|\vec{\varphi}(s)\|_{\E}^{2}
+e^{-3\gamma_{0}L},
\end{equation*}
which implies~\eqref{mon:I}.
\end{proof}

\section{Proof of Theorem~\ref{main:theo}}\label{S:4}
 In this section, we prove Theorem~\ref{main:theo} using a bootstrap argument.
 We start with a technical result that will allow us to adjust the initial value 
 with $N$ free parameters.
 \begin{lemma}[The initial unstable mode]\label{chooini}
 Let $N\ge 2$. For $n\in \{1,\ldots,N\}$, let $\sigma_{n}=\pm 1$ and $-1<\ell_{n}<1$ 
 with $-1<\ell_{1}<\ell_{2}<\cdots<\ell_{N}<1$. There exist $L_{0}\gg 1$ and $0<\delta_{0}\ll 1$ such that the following is true.
 Let $y_{1}^{0}<\cdots<y_{N}^{0}$ be such that 
 \begin{equation*}
 L=\min (y^{0}_{n+1}-y^{0}_{n},n=1,\ldots,N-1)>L_{0},
 \end{equation*}
 and $\vec{\varepsilon}\in H^{1}\times L^{2}$, 
 $\boldsymbol{a}^{+}=(a^{+}_{n})_{n\in \{1,\ldots,N\}}\in \RR^{N} $ be such that
 \begin{equation}\label{est:ini}
 \|\vec{\varepsilon}\|_{\E}<\delta<\delta_{0}\quad \mbox{and}\quad \boldsymbol{a}^{+}\in \bar{B}_{\RR^{N}}(r)\quad 
 \mbox{where}\quad r=C_{0}^{\frac{3}{4}}\left(\delta^{2}+e^{-2\gamma_{0}L}\right)^{\frac{1}{2}},
 \end{equation}
 $C_{0}$ is defined in the bootstrap~\eqref{Bootset} and to be taken large enough.
 Then, there exist $\tilde{y}_{1}^{0}<\cdots<\tilde{y}_{N}^{0}$ and $h^{+}_{1},\ldots,h_{N}^{+}\in \RR$
 satisfying
 \begin{equation}\label{estdiff}
 \sum_{n=1}^{N}\left(|h^{+}_{n}|+|{\tilde{y}}^{0}_{n}-y_{n}^{0}|\right)\le C^{\frac{13}{16}}_{0}(\delta+e^{-\gamma_{0}L}),
 \end{equation}
 such that the initial value defined by 
 \begin{equation*}
 \vec{u}_{0}=\sum_{n=1}^{N}\left(\sigma_{n}\vec{Q}_{\ell_{n}}+h_{n}^{+}\vec{Z}_{\ell_{n}}^{+}
\right)(\cdot-y^{0}_{n})+\vec{\varepsilon}
 \end{equation*}
 rewrites as:
 \begin{equation}\label{modid}
 \vec{u}_{0}=\sum_{n=1}^{N}
 \sigma_{n}\vec{Q}_{\ell_{n}}(\cdot-\tilde{y}^{0}_{n})+\vec{\varphi}(0)
 \end{equation}
 where $\vec{\varphi}(0)$ satisfies for all $n=1,\ldots,N$,
 \begin{equation}\label{defa0}
 \left(\vec{\varphi}(0),\vec{Z}_{\ell_{n}}^{0}(\cdot-\tilde{y}^{0}_{n})\right)_{L^{2}}=0,\quad a^{+}_{n}(0)=\left(\vec{\varphi}(0),\vec{Z}_{\ell_{n}}^{+}(\cdot-\tilde{y}^{0}_{n})\right)_{L^{2}}=a^{+}_{n}.
 \end{equation}
Moreover, the initial data~\eqref{modid} is modulated in the sense of Lemma~\ref{le:decom}
with $y_{n}(0)=\tilde{y}_{n}^{0}$, for all $n=1,\ldots,N$.
 \end{lemma}
\begin{proof}
Let
\begin{equation*}
\Gamma^{0}=\left(\boldsymbol{0},y_{1}^{0},\ldots,y_{N}^{0}\right)\in \R^{2N},\quad 
\Gamma=\left(h_{1}^{+},\ldots,h_{N}^{+},\tilde{y}_{1}^{0},\ldots,\tilde{y}_{N}^{0}\right)\in \R^{2N}.
\end{equation*}
Consider the map
\begin{equation*}
\begin{aligned}
\Psi&:\quad X\to \R^{2N}\\
& \left(\vec{\varepsilon},\boldsymbol{a}^{+},\Gamma\right)\mapsto \left(\Psi^{a}_{1},\ldots,\Psi^{a}_{N},\Psi_{1}^{0},\ldots,\Psi^{0}_{N}\right),
\end{aligned}
\end{equation*}
where $X=\left(H^{1}\times L^{2}\right)\times\R^{N}\times \R^{2N} $, and for $n=1,\ldots,N$,
\begin{equation*}
\begin{aligned}
\Psi_{n}^{a}
&=\sum_{n'=1}^{N}\sigma_{n'}\left(\left(\vec{Q}_{\ell_{n'}}(\cdot-y^{0}_{n'})
-\vec{Q}_{\ell_{n'}}(\cdot-\tilde{y}^{0}_{n'})\right),\vec{Z}_{\ell_{n}}^{+}(\cdot-\tilde{y}^{0}_{n})\right)_{L^{2}}\\
 &\quad +\sum_{n'=1}^{N}h_{n'}^{+}\left(\vec{Z}_{\ell_{n'}}^{+}(\cdot-y^{0}_{n'}),\vec{Z}_{\ell_{n}}^{+}(\cdot-\tilde{y}^{0}_{n})\right)_{L^{2}}+\left(\vec{\varepsilon},\vec{Z}_{\ell_{n}}^{+}(\cdot-\tilde{y}^{0}_{n})\right)_{L^{2}}-a^{+}_{n},
\end{aligned}
\end{equation*}
\begin{equation*}
\begin{aligned}
 \Psi_{n}^{0}
 &=\sum_{n'=1}^{N}\sigma_{n'}\left(\left(\vec{Q}_{\ell_{n'}}(\cdot-y^{0}_{n'})
 -\vec{Q}_{\ell_{n'}}(\cdot-\tilde{y}^{0}_{n'})\right),\vec{Z}_{\ell_{n}}^{0}(\cdot-\tilde{y}^{0}_{n})\right)_{L^{2}}\\
 &\quad +\sum_{n'=1}^{N}h_{n'}^{+}\left(\vec{Z}_{\ell_{n'}}^{+}(\cdot-y^{0}_{n'}),\vec{Z}_{\ell_{n}}^{0}(\cdot-\tilde{y}^{0}_{n})\right)_{L^{2}}+\left(\vec{\varepsilon},\vec{Z}_{\ell_{n}}^{0}(\cdot-\tilde{y}^{0}_{n})\right)_{L^{2}}.
\end{aligned}
\end{equation*}
From~\eqref{modid},
\begin{equation}\label{eq:vp0}
\vec{\varphi}(0)=\sum_{n=1}^{N}\sigma_{n}\left(\vec{Q}_{\ell_{n}}(\cdot-y^{0}_{n})
-\vec{Q}_{\ell_{n}}(\cdot-\tilde{y}^{0}_{n})\right)+\sum_{n=1}^{N}h_{n}^{+}\vec{Z}_{\ell_{n}}^{+}(\cdot-y^{0}_{n})+\vec{\varepsilon},
\end{equation} 
and thus the set of conditions in~\eqref{defa0} is equivalent to $\Psi\left(\vec{\varepsilon},\boldsymbol{a}^{+},\Gamma\right)=\boldsymbol{0}\in \R^{2N}$.
We solve this nonlinear system by the Implicit Function Theorem. First, it is easy to check that 
\begin{equation*}
\Psi(\vec{0},\boldsymbol{0},\Gamma^{0})=0.
\end{equation*}
Second, by direct computation and integration by parts, for any $(\vec{\varepsilon},\boldsymbol{a}^{+},\Gamma)$ around $\left(\vec{0},\boldsymbol{0},\Gamma^{0}\right)$,
\begin{equation*}
D_{\Gamma}\Psi(\vec{\varepsilon},\boldsymbol{a}^{+},\Gamma)=\left(\begin{array}{cc}
A& C\\
B& D
\end{array}\right),
\end{equation*}
where
\begin{align*}
A&=\left(\bigg(\vec{Z}_{\ell_{n}}^{+}(\cdot-y^{0}_{n}),\vec{Z}_{\ell_{n'}}^{+}(\cdot-{y}^{0}_{n'})\bigg)_{L^{2}}\right)_{n,n'\in\{1,\ldots,N\}}+O\left(\|\vec{\varepsilon}\|_{\E}+|\Gamma-\Gamma^{0}|\right),\\
 B&=\left(\bigg(\sigma_{n}\vec{Z}_{\ell_{n}}^{0}(\cdot-y^{0}_{n}),\vec{Z}_{\ell_{n'}}^{+}(\cdot-{y}^{0}_{n'})\bigg)_{L^{2}}\right)_{n,n'\in\{1,\ldots,N\}}+O\left(\|\vec{\varepsilon}\|_{\E}+|\Gamma-\Gamma^{0}|\right),\\
 C&=\left(\bigg(\sigma_{n}\vec{Z}_{\ell_{n}}^{0}(\cdot-y^{0}_{n}),\vec{Z}_{\ell_{n'}}^{+}(\cdot-{y}^{0}_{n'})\bigg)_{L^{2}}\right)_{n,n'\in\{1,\ldots,N\}}+O\left(\|\vec{\varepsilon}\|_{\E}+|\Gamma-\Gamma^{0}|\right),\\
 D&=\left(\bigg(\sigma_{n}\vec{Z}_{\ell_{n}}^{0}(\cdot-y^{0}_{n}),\sigma_{n'}\vec{Z}_{\ell_{n'}}^{0}(\cdot-{y}^{0}_{n'})\bigg)_{L^{2}}\right)_{n,n'\in\{1,\ldots,N\}}+O\left(\|\vec{\varepsilon}\|_{\E}+|\Gamma-\Gamma^{0}|\right).
 \end{align*}
 Moreover, from $\left(\vec{Z}_{\ell}^{0},\vec{Z}_{\ell}^{+}\right)_{L^{2}}=0$ and similar argument as in the proof of~\eqref{tech1}, we obtain
\begin{align*}
 A&={\rm{diag}}\left(\left(\vec{Z}_{\ell_{1}}^{+},\vec{Z}_{\ell_{1}}^{+}\right)_{L^{2}},\ldots, \left(\vec{Z}_{\ell_{N}}^{+},\vec{Z}_{\ell_{N}}^{+}\right)_{L^{2}}\right)+
 O\left(\|\vec{\varepsilon}\|_{\E}+|\Gamma-\Gamma^{0}|+e^{-2\gamma_{0}L}\right),\\
 B&=O\left(\|\vec{\varepsilon}\|_{\E}+|\Gamma-\Gamma^{0}|+e^{-2\gamma_{0}L}\right),\\ 
 C&=O\left(\|\vec{\varepsilon}\|_{\E}+|\Gamma-\Gamma^{0}|+e^{-2\gamma_{0}L}\right),\\
 D&={\rm{diag}}\left(\left(\vec{Z}_{\ell_{1}}^{0},\vec{Z}_{\ell_{1}}^{0}\right)_{L^{2}},\ldots, \left(\vec{Z}_{\ell_{N}}^{0},\vec{Z}_{\ell_{N}}^{0}\right)_{L^{2}}\right)+O\left(\|\vec{\varepsilon}\|_{\E}+|\Gamma-\Gamma^{0}|+e^{-2\gamma_{0}L}\right).
 \end{align*}
 Thus, $D_{\Gamma}\Psi(\vec{\varepsilon},\boldsymbol{a}^{+},\Gamma)$ is an invertible matrix for $L>L_{0}$ large enough, with a lower bound uniform around $\left(\vec{0},\boldsymbol{0},\Gamma^{0}\right)$. Therefore, by the uniform variant of the implicit function theorem, there exist $0<\delta_{1}\ll 1$
 (independent with the choose of $y_{1}^{0}<\cdots<y_{N}^{0}$) and $C^{1}$ map 
 \begin{equation*}
 \Pi:\ B_{H^{1}\times L^{2}}(\vec{0},\delta_{1})\times B_{\R^{N}}(\boldsymbol{0},\delta_{1})
 \to B_{\R^{2N}}(\Gamma^{0},\delta_{1}),
 \end{equation*}
 such that for all $\vec{\varepsilon}\in B_{H^{1}\times L^{2}}(\vec{0},\delta_{1})$,
 $\boldsymbol{a}^{+}\in B_{\R^{N}}(\boldsymbol{0},\delta_{1})$ and $\Gamma\in B_{\R^{2N}}(\Gamma^{0},\delta_{1})$,
 \begin{equation*}
 \Psi\left(\vec{\varepsilon},\boldsymbol{a}^{+},\Gamma\right)=0\quad\mbox{if and only if}\quad 
 \Gamma=\Pi\left(\vec{\varepsilon},\boldsymbol{a}^{+}\right).
 \end{equation*}
 Moreover, taking $0<\delta<\delta_{0}\ll 1$ small enough, $L>L_{0}$ large enough, for any $(\vec{\varepsilon},\boldsymbol{a}^{+})$ satisfying~\eqref{est:ini}, we have 
 \begin{equation*}
 \left|\Gamma-\Gamma^{0}\right|=\left|\Pi\left(\vec{\varepsilon},\boldsymbol{a}^{+}\right)-\Pi\left(\vec{0},\boldsymbol{0}\right)\right|\lesssim \|\vec{\varepsilon}\|_{H^{1}\times L^{2}}+\|\boldsymbol{a}^{+}\|
 \lesssim \delta+C_{0}^{\frac{3}{4}}\left(\delta^{2}+e^{-2\gamma_{0}L}\right)^{\frac{1}{2}},
 \end{equation*}
 which implies~\eqref{estdiff} for taking $C_{0}$ large enough (independent with $\delta_{0}$ and $L_{0}$). 
 The proof of Lemma~\ref{chooini} is complete.
\end{proof}

We are in a position to complete the proof of Theorem~\ref{main:theo}.
\begin{proof}[Proof of Theorem~\ref{main:theo}]
 Let $\vec{\varepsilon}\in H^1\times L^2$ 
 and $y_{1}^{0}<\cdots<y_{N}^{0}$ as in the statement of Theorem~\ref{main:theo}.
 For all 
$\boldsymbol{a}^{+}(0)=\boldsymbol{a}^{+}=(a^{+}_{1},\ldots,a^{+}_{N})\in \bar{B}_{\RR}(r)$, we consider the solution $\vec{u}=(u_{1},u_{2})$ with the initial data
as defined in Lemma~\ref{chooini}
\begin{equation}\label{def:u0}
\vec{u}_{0}=\sum_{n=1}^{N}\left(\sigma_{n}\vec{Q}_{\ell_{n}}+h_{n}^{+}\vec{Z}_{\ell_{n}}^{+}
\right)(\cdot-y^{0}_{n})+\vec{\varepsilon}=\sum_{n=1}^{N}
\sigma_{n}\vec{Q}_{\ell_{n}}(\cdot-\tilde{y}^{0}_{n})+\vec{\varphi}(0).
\end{equation}
Note that, for the proof of Theorem~\ref{main:theo}, we just need to prove the existence of $\vec{u}_{0}$ such that $T_{*}(\vec{u}_{0})=\infty$. We start by closing all bootstrap estimates except the one for the instable modes.
Last, we prove the existence of suitable parameters $\boldsymbol{a}^{+}=(a_{1}^{+},\ldots,a_{N}^{+})$ using a topological argument. Denote $q_{0}=\min(3,q+1)$.

{\bf{Step 1.}} Closing the estimates in $\vec{\varphi}$. First, from~\eqref{estdiff} and~\eqref{def:u0},
\begin{equation}\label{est:vp0}
\begin{aligned}
\|\vec{\varphi}(0)\|^{2}_{\E}
&\lesssim \sum_{n=1}^{N}\|\vec{Q}_{\ell_{n}}(\cdot-y_{n}^{0})-\vec{Q}_{\ell_{n}}(\cdot-\tilde{y}_{n}^{0})\|_{\E}^{2}
+\sum_{n=1}^{N}\|h_{n}^{+}\vec{Z}_{\ell_{n}}^{+}(\cdot-y^{0}_{n})\|_{\E}^{2}+\|\vec{\varepsilon}\|_{\E}^{2}\\
&\lesssim \|\vec{\varepsilon}\|^{2}_{\E}+ \sum_{n=1}^{N}\left(|h^{+}_{n}|^{2}+|{\tilde{y}}^{0}_{n}-y_{n}^{0}|^{2}\right)\lesssim C^{\frac{13}{8}}_{0}(\delta^{2}+e^{-2\gamma_{0}L}).
\end{aligned}
\end{equation} 
Second, from~\eqref{mon:I}, conservation of energy $E(\vec{u}(t))$ and momentum $\mathcal{I}(\vec{u}(t))$, for any $t\in [0,T_{*}(\vec{u}_{0}(0))]$,
\begin{align*}
\mathcal{E}(\vec{u}(t))- \mathcal{E}(\vec{u}(0))
&=\sum_{n=2}^{N}c_{n}\left(\mathcal{J}_{n}(t)-\mathcal{J}_{n}(0)\right)\\
&\le C_{1}\left(\sum_{n=2}^{N}c_{n}\right)\left(\frac{1}{L^{2\alpha-1}}\sup_{s\in[0,t]}\|\vec{\varphi}(s)\|^{2}_{\E}+e^{-3\gamma_{0}L}\right).
\end{align*}
Note that, from~\eqref{expan:E},
\begin{align*}
&\sum_{n=1}^{N}\tilde{c}_{n}H_{n}(\vec{\varphi}(t),\vec{\varphi}(t))\\
&\quad =\mathcal{E}(\vec{u}(t))- \mathcal{E}(\vec{u}(0))
+O\left(\frac{\|\vec{\varphi}(t)\|^{2}_{\E}}{L}+
\|\vec{\varphi}(t)\|_{\E}^{q_{0}}+e^{-3\gamma_{0} L}+\|\vec{\varphi}(0)\|_{\E}^{2}\right).
\end{align*}
Therefore, using~\eqref{Bootset},~\eqref{coer:HN},~\eqref{est:vp0} and $\frac{1}{2}<\alpha<\frac{4}{7}$,
\begin{align*}
\mu\|\vec{\varphi}(t)\|_{\E}^{2}
& \lesssim \frac{1}{L^{2\alpha-1}}\sup_{s\in [0,t]}\|\vec{\varphi}(s)\|_{\E}^{2}
+\frac{1}{L}\|\vec{\varphi}(t)\|_{\E}^{2}+\|\vec{\varphi}(t)\|_{\E}^{q_{0}}
\\
&\quad +\left(C_{0}^{\frac{3}{2}}+C_{0}^{\frac{7}{4}}+C_{0}^{\frac{13}{8}}\right)(\delta^{2}+e^{-2\gamma_0L})+e^{-3\gamma_0 L}\\
& \lesssim \left(\frac{C^{2}_{0}}{L^{2\alpha-1}}+C_{0}^{\frac{7}{4}}\right)\left(\delta^{2}+e^{-2\gamma_{0}L}\right)+C_{0}^{q_{0}}\left(\delta^{q_{0}}+e^{-q_{0}\gamma_{0}L}\right)+e^{-3\gamma_0 L}.
\end{align*}
This strictly improves the estimate on $\vec{\varphi}$ in~\eqref{Bootset} for taking $L$, $C_{0}$ large enough and $\delta$ small enough.

{\bf{Step 2.}} Closing the estimates in $\min_{n}(y_{n+1}-y_{n})$. From~\eqref{est:yt},~\eqref{estdiff} and $\gamma_{1}=\frac{1}{8}\min_{n}(\ell_{n+1}-\ell_{n})$, for any $n=1,\ldots,N-1$ and $t\in[0,T_{*}(\vec{u}_{0})]$,
\begin{align*}
y_{n+1}(t)-y_{n}(t)&\ge (\ell_{n+1}-\ell_{n})t+(y_{n+1}^{0}-y^{0}_{n})-2\gamma_{1}t-|\tilde{y}_{n+1}^{0}-y^{0}_{n+1}|-|\tilde{y}_{n}^{0}-y^{0}_{n}|\\
&\ge 8\gamma_{1}t+L-2\gamma_{1}t+O\left(C_{0}^{\frac{13}{16}}(\delta+e^{-\gamma_0 L})\right),
\end{align*}
which strictly improves the estimate on $ \min_{n}(y_{n+1}-y_{n})$ in~\eqref{Bootset} for taking $L$ large enough and $\delta$ small enough.

{\bf{Step 3.}} Closing the estimates in $\boldsymbol{a}^{-}=(a^{-}_{n})_{n\in \{1,\ldots,N\}}$. Note that, from~\eqref{est:vp0},
\begin{equation*}
\sum_{n=1}^{N}(a^{-}_{n}(0))^{2}\lesssim \|\vec{\varphi}(0)\|_{\E}^{2}\lesssim C^{\frac{13}{8}}_{0}(\delta^{2}+e^{-2\gamma_{0}L}).
\end{equation*}
By direct computation,~\eqref{ode:z},~\eqref{Bootset} and~\eqref{est:theta}, for $n=1,\ldots,N$,
\begin{align*}
\frac{\d}{\d t}\left(e^{2\alpha_{n}t}(a^{-}_{n}(t))^{2}\right)
&=2e^{2\alpha_{n}t}a_{n}^{-}(t)\left(\frac{\d}{\d t}a^{-}_{n}(t)+\alpha_{n}a^{-}_{n}(t)\right)\\
&=e^{2\alpha_{n}t}O\left(|a_{n}^{-}|^{q_{0}}+\|\vec{\varphi}\|_{\E}^{q_{0}}+\theta\right)\\
&=e^{2\alpha_{n}t}O\left(C_{0}^{q_{0}}\left(\delta^{q_{0}}+e^{-q_{0}\gamma_{0}L}\right)+e^{-3\gamma_{0} L}\right).
\end{align*}
Integrating on $[0,t]$, for any $t\in [0,T_{*}(\vec{u}_{0})]$ and any $n=1,\ldots,N$, we obtain
\begin{align*}
(a_{n}^{-}(t))^{2}
&\lesssim e^{-2\alpha_{n}t}(a_{n}^{-}(0))^{2}+
e^{-2\alpha_{n}t}\int_{0}^{t}e^{2\alpha_{n}s}\left(C_{0}^{q_{0}}\left(\delta^{q_{0}}+e^{-q_{0}\gamma_{0}L}\right)+e^{-3\gamma_{0} L}\right)\d s\\
&\lesssim C^{\frac{13}{8}}_{0}(\delta^{2}+e^{-2\gamma_{0}L})+C_{0}^{q_{0}}\left(\delta^{q_{0}}+e^{-q_{0}\gamma_{0}L}\right)
+e^{-3\gamma_{0} L},
\end{align*}
which strictly improves the estimate on ${\boldsymbol{a}}^{-}=(a^{-}_{n})_{n\in \{1,\ldots,N\}}$ in~\eqref{Bootset} for taking $L$ large enough and $\delta$ small enough.

{\bf{Step 4.}} Final argument on the unstable parameters. Let
\begin{equation*}
b(t)=\sum_{n=1}^{N}(a^{+}_{n}(t))^{2}\quad \mbox{and}\quad \bar{\alpha}=\min_{n}\alpha_{n}.
\end{equation*}
We claim, for any time $t\in [0,T_{*}(\vec{u}_{0})]$ where it holds $b(t)=C_{0}^{\frac{3}{2}}\left(\delta^{2}+e^{-2\gamma_{0}L}\right)$, the following transversality property holds,
\begin{equation}\label{tran}
\frac{\d}{\d t}b(t)\ge \bar{\alpha}C_{0}^{\frac{3}{2}}\left(\delta^{2}+e^{-2\gamma_{0}L}\right)>0.
\end{equation}
 From~\eqref{ode:z},~\eqref{Bootset} and~\eqref{est:theta}, taking $C_{0}$ large enough, for any time $t\in [0,T_{*}]$ where it holds $b(t)=C_{0}^{\frac{3}{2}}\left(\delta^{2}+e^{-2\gamma_{0}L}\right)$, 
 \begin{equation*}
 \begin{aligned}
 \frac{\d}{\d t}b(t)
 &=2\sum_{n=1}^{N}\alpha_{n}(a^{+}_{n}(t))^{2}+O\left[\bigg(\sum_{n=1}^{N}|a_{n}^{+}(t)|\bigg)\bigg(\|\vec{\varphi}\|_{\E}^{q_{0}-1}+\theta\bigg)\right]\\
 &=2\sum_{n=1}^{N}\alpha_{n}(a^{+}_{n}(t))^{2}+O\left(\sum_{n=1}^{N}|a_{n}^{+}|^{q_{0}}+
 \|\vec{\varphi}\|_{\E}^{q_{0}}+e^{-3\gamma_{0}L}\right)\\
 &\ge 2\bar{\alpha}C_{0}^{\frac{3}{2}}\left(\delta^{2}+e^{-2\gamma_{0}L}\right)
 -C_{0}^{q_{0}+1}\left(\delta^{q_{0}}+e^{-q_{0}\gamma_{0}L}\right)-C_{0}e^{-3\gamma_0 L},
 \end{aligned}
 \end{equation*}
 which implies~\eqref{tran} for taking $\delta$ small enough and $L$ large enough (depending on $C_{0}$). The transversality relation~\eqref{tran} is enough to justify the 
 existence of at least a couple 
 ${\boldsymbol{a}}^{+}(0)=(a^{+}_{1}(0),\ldots,a^{+}_{N}(0))\in \bar{\mathcal{B}}_{\RR^{N}}(r)$ such that $T_{*}(\vec{u}_{0})=\infty$ where $r=C_{0}^{\frac{3}{4}}\left(\delta^{2}+e^{-2\gamma_{0}L}\right)^{\frac{1}{2}}$.

The proof is by contradiction, we assume that for all $\boldsymbol{a}^{+}(0)\in \bar{\mathcal{B}}_{\RR^{N}}(r)$, it holds $T_{*}(\vec{u}_{0})<\infty$. Then, a contradiction follows from the following discussion (see for instance more details in~\cite{CMM} and~\cite[Section 3.1]{CMkg}).

\emph{Continuity of $T_{*}(\vec{u}_{0})$.} The above transversality condition~\eqref{tran} implies that the map
\begin{equation*}
\boldsymbol{a}^{+}(0)\in \bar{\mathcal{B}}_{\RR^{N}}(r)\mapsto T_{*}(\vec{u}_{0})\in [0,\infty)
\end{equation*}
is continuous and 
\begin{equation*}
T_{*}(\vec{u}_{0})=0\quad \mbox{for}\quad \boldsymbol{a}^{+}(0)\in \mathcal{S}_{\RR^{N}}\left(r\right).
\end{equation*}

\emph{Construction of a retraction}. We define
\begin{align*}
\mathcal{M}:\ \bar{\mathcal{B}}_{\RR^{N}}(r)&\mapsto {\mathcal{S}}_{\RR^{N}}(r)\\
\boldsymbol{a}^{+}(0)&\mapsto \boldsymbol{a}^{+}(T_{*}(\vec{u}_{0})).
\end{align*}
From what precedes, $\mathcal{M}$ is continuous. Moreover, $\mathcal{M}$ restricted to
${\mathcal{S}}_{\RR^{N}}(r)$ is the identity. The
existence of such a map is contradictory with the no retraction theorem for continuous maps from the ball to
the sphere.

\smallskip

We have proved the existence of $\boldsymbol{a}^{+}(0)\in \bar{\mathcal{B}}_{\RR^{N}}(r)$, associated to a global solution $\vec{u}=(u_{1},u_{2})$ of~\eqref{NLKGvec} with initial data defined in Lemma~\ref{chooini}, which also satisfies~\eqref{estQ0} and~\eqref{Bootset} for all $t\in [0,\infty)$. The proof of Theorem~\ref{main:theo} is  thus complete.
\end{proof}

\end{document}